\documentclass{article}

\usepackage{hyperref}
\usepackage{amsfonts,amsmath,amssymb,amscd,amsthm}
\usepackage{latexsym}
\usepackage{undertilde}
\usepackage[all]{xy}
\usepackage{graphicx}
\usepackage{epsfig}

\input{epsf}

\newcommand{\Cat}{{\rm Cat}}

\newcommand{\Vect}{{\rm Vect}}
\newcommand{\Rep}{{\rm Rep}}
\newcommand{\SES}{{\rm SES}}
\newcommand{\FinVect}{{\rm FinVect}}

\newcommand{\Span}{{\rm Span}}

\renewcommand{\hom}{{\rm hom}}
\newcommand{\To}{\Rightarrow}
\renewcommand{\to}{\rightarrow}
\newcommand{\tensor}{\otimes}

\newcommand{\maps}{\colon}

\newcommand{\iso}{\cong}

\newcommand{\tr}{{\rm tr}}

\newcommand{\SL}{\mathrm{SL}}
\newcommand{\g}{{\mathfrak g}}

\newcommand{\R}{{\mathbb R}}
\newcommand{\C}{{\mathbb C}}
\newcommand{\F}{{\mathbb F}}
\newcommand{\N}{{\mathbb N}}
\newcommand{\Q}{{\mathbb Q}}
\newcommand{\Z}{{\mathbb Z}}
\newcommand{\U}{{\rm U}}

\newcommand{\Aut}{{\rm Aut}}

\newcommand{\ip}[2]{\langle #1,#2 \rangle}
\newcommand{\Over}{/\!/}
\newcommand{\w}{\! \colon \!}

\newcommand{\Ob}{{\rm Ob}}
\newcommand{\Mor}{{\rm Mor}}
\newcommand{\comment}[1]{}
\renewcommand{\u}[1]{\underline{#1}}

\newtheorem{thm}{Theorem}
\newtheorem{definition}[thm]{Definition}
\newtheorem{lemma}[thm]{Lemma}

\newtheorem{proposition}[thm]{Proposition}
\newtheorem{example}[thm]{Example}

\newtheorem*{theorem*}{Theorem}
\newtheorem*{definition*}{Definition}
\newtheorem*{corollary*}{Corollary}
\newtheorem*{proposition*}{Proposition}
\newtheorem*{example*}{Example}

\newdir{ >}{{}*!/-7pt/@{>}}
\newdir{> >}{*!/0pt/@{>}*!/-5pt/@{>}}

\begin{document}
\title{Higher Dimensional Algebra VII: Groupoidification}
\author{John C.\ Baez, Alexander E.\ Hoffnung, and Christopher D.\ Walker\\
       Department of Mathematics, University of California\\
       Riverside, CA 92521 USA}
\date{August 29, 2009}
\maketitle

\begin{abstract}
\noindent
Groupoidification is a form of categorification in which vector spaces
are replaced by groupoids and linear operators are replaced by spans
of groupoids.  We introduce this idea with a detailed exposition of
`degroupoidification': a systematic process that turns groupoids and
spans into vector spaces and linear operators.  Then we present three
applications of groupoidification.  The first is to Feynman diagrams.
The Hilbert space for the quantum harmonic oscillator arises naturally
from degroupoidifying the groupoid of finite sets and bijections.
This allows for a purely combinatorial interpretation of creation and
annihilation operators, their commutation relations, field operators,
their normal-ordered powers, and finally Feynman diagrams.  The second
application is to Hecke algebras.  We explain how to groupoidify the
Hecke algebra associated to a Dynkin diagram whenever the deformation
parameter $q$ is a prime power.  We illustrate this with the simplest
nontrivial example, coming from the $A_2$ Dynkin diagram.  In this
example we show that the solution of the Yang--Baxter equation built
into the $A_2$ Hecke algebra arises naturally from the axioms of
projective geometry applied to the projective plane over the finite
field $\mathbb{F}_q$.  The third application is to Hall algebras.  We
explain how the standard construction of the Hall algebra from the
category of $\mathbb{F}_q$ representations of a simply-laced quiver
can be seen as an example of degroupoidification.  This in turn
provides a new way to categorify---or more precisely, groupoidify---the 
positive part of the quantum group associated to the quiver.
\end{abstract}

\section{Introduction}\label{intro}

`Groupoidification' is an attempt to expose the combinatorial
underpinnings of linear algebra---the hard bones of set theory
underlying the flexibility of the continuum.  One of the main lessons
of modern algebra is to avoid choosing bases for vector spaces until
you need them.  As Hermann Weyl wrote, ``The introduction of a
coordinate system to geometry is an act of violence".  But vector
spaces often come equipped with a natural basis---and when this
happens, there is no harm in taking advantage of it.  The most obvious
example is when our vector space has been defined to consist of formal
linear combinations of the elements of some set.  Then this set is our
basis.  But surprisingly often, the elements of this set are {\it
isomorphism classes of objects in some groupoid}.  This is when
groupoidification can be useful.  It lets us work directly with the
groupoid, using tools analogous to those of linear algebra, without
bringing in the real numbers (or any other ground field).

For example, let $E$ be the groupoid of finite sets and bijections.
An isomorphism class of finite sets is just a natural number, so the
set of isomorphism classes of objects in $E$ can be identified with
$\N$.  Indeed, this is why natural numbers were invented in the first
place: to count finite sets.  The real vector space with $\N$ as basis
is usually identified with the polynomial algebra $\R[z]$, since that
has basis $z^0, z^1, z^2, \dots.$ Alternatively, we can work with {\it
infinite} formal linear combinations of natural numbers, which form
the algebra of formal power series, $\R[[z]]$.  So, the vector space
of formal power series is a kind of stand-in for the groupoid of
finite sets.

Indeed, formal power series have long been used as `generating
functions' in combinatorics \cite{Wilf}.  Given any combinatorial
structure we can put on finite sets, its generating function is the
formal power series whose $n$th coefficient says how many ways we can
put this structure on an $n$-element set.  Andr\'e Joyal formalized
the idea of `a structure we can put on finite sets' in terms of {\it
esp\`eces de structures}, or `structure types' \cite{BLL,Joyal,
Joyal2}.  Later his work was generalized to `stuff types'
\cite{BaezDolan:2001}, which are a key example of groupoidification.

Heuristically, a stuff type is a way of equipping finite sets with a
specific type of extra stuff---for example a 2-coloring, or a linear
ordering, or an additional finite set.  Stuff types have generating
functions, which are formal power series.  Combinatorially interesting
operations on stuff types correspond to interesting operations on their
generating functions: addition, multiplication, differentiation, and
so on.  Joyal's great idea amounts to this: {\it work directly with
stuff types as much as possible, and put off taking their generating
functions.}  As we shall see, this is an example of groupoidification.

To see how this works, we should be more precise.  A {\bf stuff type}
is a groupoid over the groupoid of finite sets: that is, a groupoid
$\Psi$ equipped with a functor $v \maps \Psi \to E$.  The reason for
the funny name is that we can think of $\Psi$ as a groupoid of finite
sets `equipped with extra stuff'.  The functor $v$ is then the
`forgetful functor' that forgets this extra stuff and gives the
underlying set.

The generating function of a stuff type $v \maps \Psi \to E$ is the
formal power series
\begin{equation}
\label{eq:generating_function}
       \utilde{\Psi}(z) = \sum_{n = 0}^\infty |v^{-1}(n)| \, z^n .
\end{equation}
Here $v^{-1}(n)$ is the `full inverse image' of any $n$-element
set, say $n \in E$.  We define this term later, but the idea is
straightforward: $v^{-1}(n)$ is the groupoid of $n$-element sets
equipped with the given type of stuff.  The $n$th coefficient of the
generating function measures the size of this groupoid.  

But how?  Here we need the concept of {\it groupoid cardinality}.
It seems this concept first appeared in algebraic geometry
\cite{Behrends:1993, Kim:1995}.  We rediscovered it by
pondering the meaning of division \cite{BaezDolan:2001}.  Addition of
natural numbers comes from disjoint union of finite sets, since
\[        |S + T| = |S| + |T|  .\]
Multiplication comes from cartesian product:
\[        |S \times T| = |S| \times |T|  .\]
But what about division?

If a group $G$ acts on a set $S$, we can `divide' the set by the group
and form the quotient $S/G$.  If $S$ and $G$ are finite and $G$ acts
freely on $S$, $S/G$ really deserves the name `quotient', since then
\[          |S/G| = |S|/|G|.   \]
Indeed, this fact captures some of our naive intuitions about division.
For example, why is $6/2 = 3$?  We can take a 6-element
set $S$ with a free action of the group $G = \Z/2$ and 
construct the set of orbits $S/G$:

\medskip
\centerline{\epsfysize=1.0in\epsfbox{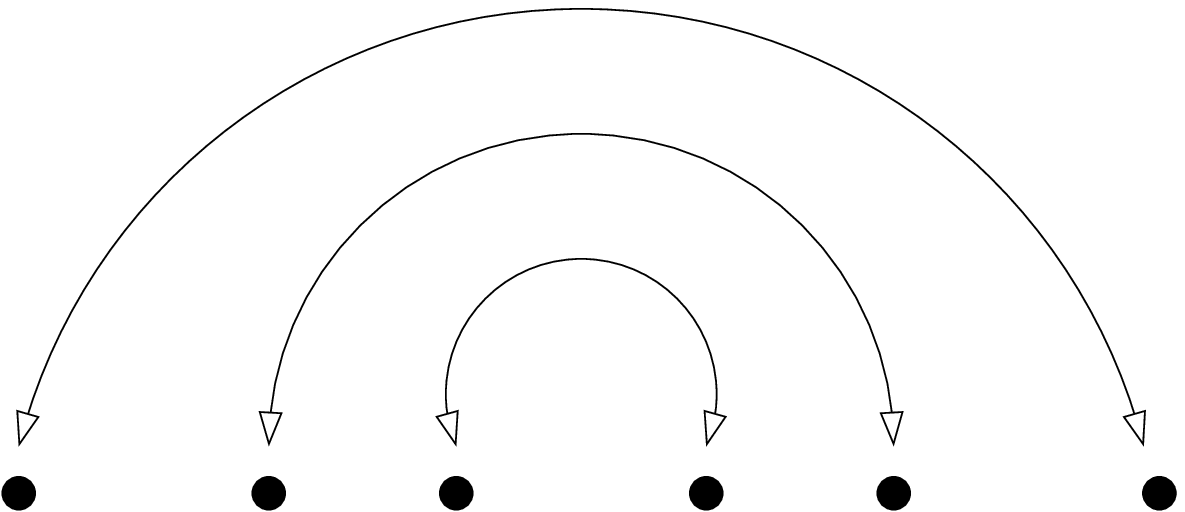}}
\medskip

\noindent
Since we are `folding the 6-element set in half', we get $|S/G| = 3$.

The trouble starts when the action of $G$ on $S$ fails to
be free.  Let's try the same trick starting with a 5-element set:

\medskip
\centerline{\epsfysize=1.0in\epsfbox{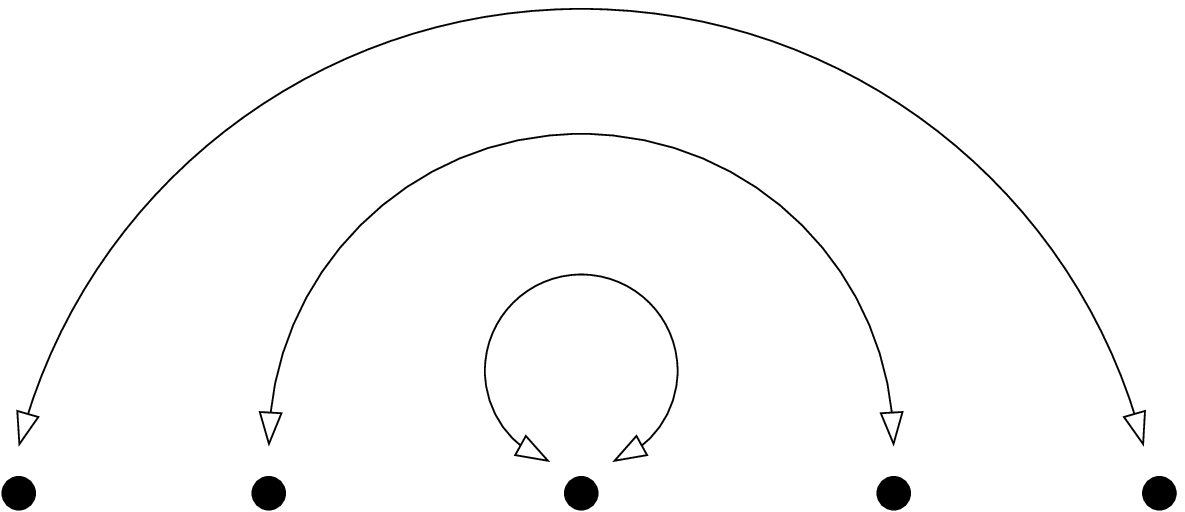}}
\medskip

\noindent
We don't obtain a set with $2\frac{1}{2}$ elements!  The reason is
that the point in the middle gets mapped to itself.  To get the
desired cardinality $2 \frac{1}{2}$, we would need a way to count this
point as `folded in half'.

To do this, we should first replace the ordinary quotient $S/G$ by the
`action groupoid' or \textbf{weak quotient} $S/\!/G$.  This is the
groupoid where objects are elements of $S$, and a morphism from $s \in
S$ to $s' \in S$ is an element $g \in G$ with $gs = s'$.  Composition
of morphisms works in the obvious way.  Next, we should define the
`cardinality' of a groupoid as follows.  For each isomorphism class of
objects, pick a representative $x$ and compute the reciprocal of the
number of automorphisms of this object; then sum the result over
isomorphism classes.  In other words, define the \textbf{cardinality} of
a groupoid $X$ to be
\begin{equation}
\label{eq:groupoid_cardinality}
      |X| = \sum_{\rm isomorphism\; classes \; of \; objects\; [x]}
\frac{1}{|\Aut(x)|} \;. 
\end{equation}
With these definitions, our problematic example gives a groupoid
$S/\!/G$ with cardinality $2\frac{1}{2}$, since the point in the
middle of the picture gets counted as `half a point'.  In fact,
\[       |S/\!/G| = |S|/|G|    \]
whenever $G$ is a finite group acting on a finite set $S$.  

The concept of groupoid cardinality gives an elegant definition of the
generating function of a stuff type---Equation
\ref{eq:generating_function}---which matches the usual `exponential
generating function' from combinatorics.  For the details of how this
works, see Example \ref{generating_function}.

Even better, we can vastly generalize the notion of generating
function, by replacing $E$ with an arbitrary groupoid.  For any
groupoid $X$ we get a vector space: namely $\R^{\underline{X}}$, the
space of functions $\psi \maps \underline{X} \to \R$, where
$\underline{X}$ is the set of isomorphism classes of objects in $X$.
Any sufficiently nice groupoid over $X$ gives a vector in this
vector space.

The question then arises: what about linear operators?  Here it is 
good to take a lesson from Heisenberg's matrix mechanics.  In his
early work on quantum mechanics, Heisenberg did not know about matrices.
He reinvented them based on this idea: a matrix $S$ can describe
a quantum process by letting the matrix entry $S_i^j \in \C$ stand
for the `amplitude' for a system to undergo a transition from its 
$i$th state to its $j$th state.  

The meaning of complex amplitudes was somewhat mysterious---and indeed
it remains so, much as we have become accustomed to it.  However, the
mystery evaporates if we have a matrix whose entries are natural
numbers.  Then the matrix entry $S_i^j \in \N$ simply counts the {\it
number of ways} for the system to undergo a transition from its $i$th
state to its $j$th state.

Indeed, let $X$ be a set whose elements are possible `initial states'
for some system, and let $Y$ be a set whose elements are possible
`final states'.  Suppose $S$ is a set equipped with maps to $X$ and
$Y$:
\[\xymatrix{
 & S\ar[dl]_q\ar[dr]^p & \\
 Y & & X \\
}\]
Mathematically, we call this setup a span of sets.  Physically,
we can think of $S$ as a set of possible `events'.  Points in $S$ sitting 
over $i \in X$ and $j \in Y$ form a subset
\[        S_i^j = \{ s \colon \; q(s) = j, \; p(s) = i \}  . \]
We can think of this as the {\it set of ways} for the system to undergo a 
transition from its $i$th state to its $j$th state.  
Indeed, we can picture $S$ more vividly as a matrix of sets: \\ \\

\[
\begin{picture}(230,125)(0,25)
 \includegraphics[scale=0.5]{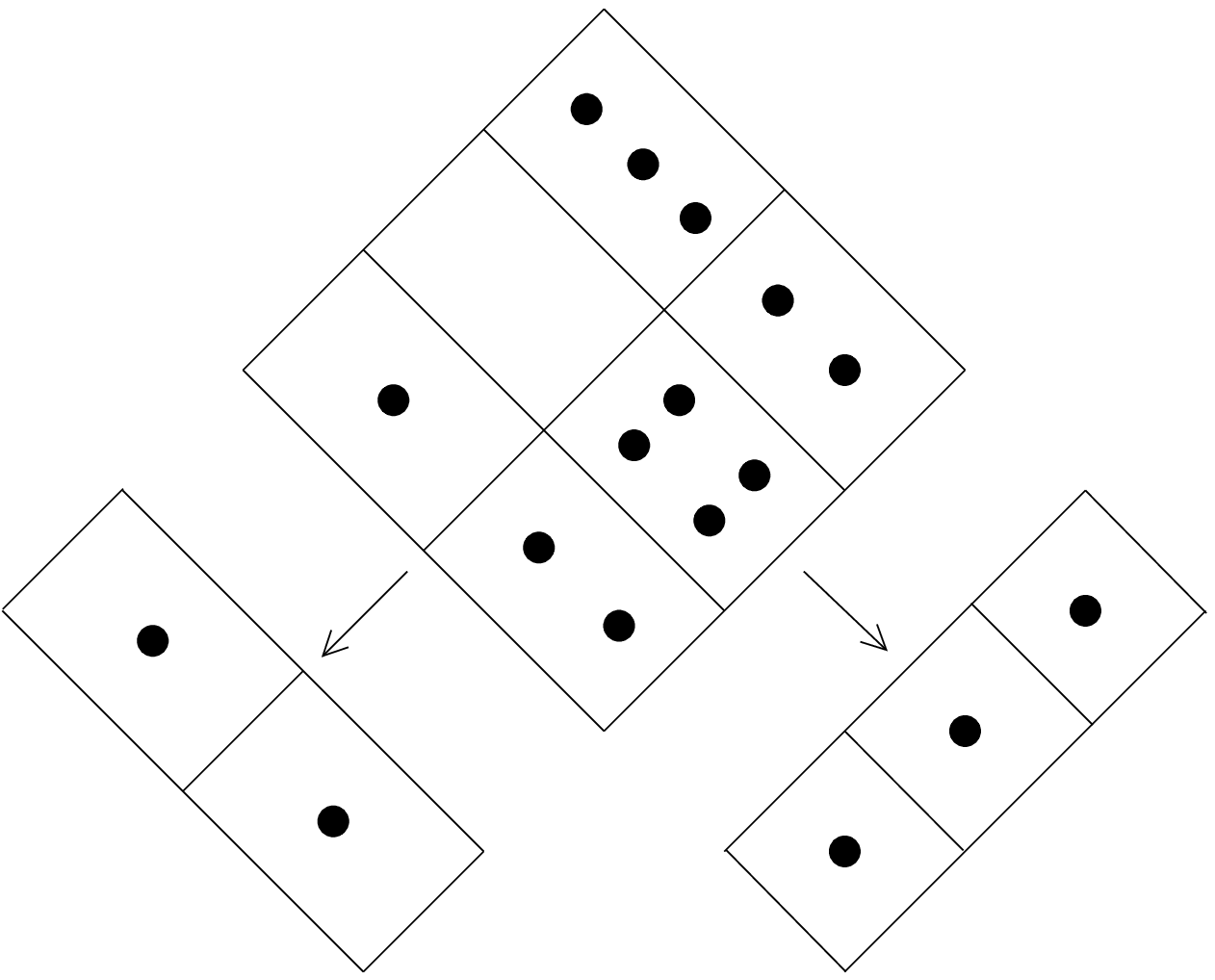}
 \put(-138,88){$q$}
 \put(-52,88){$p$}
 \put(-8,98){$X$}
 \put(-180,98){$Y$}
 \put(-80,170){$S$}
\end{picture}
\]
\noindent
If all the sets $S_i^j$ are finite, we get a matrix of natural
numbers $|S_i^j|$.

Of course, matrices of natural numbers only allow us to do a limited
portion of linear algebra.  We can go further if we consider, not spans
of sets, but {\it spans of groupoids}.  We can picture one of
these roughly as follows:

\[
\begin{picture}(230,145)(0,30)
 \includegraphics[scale=0.5]{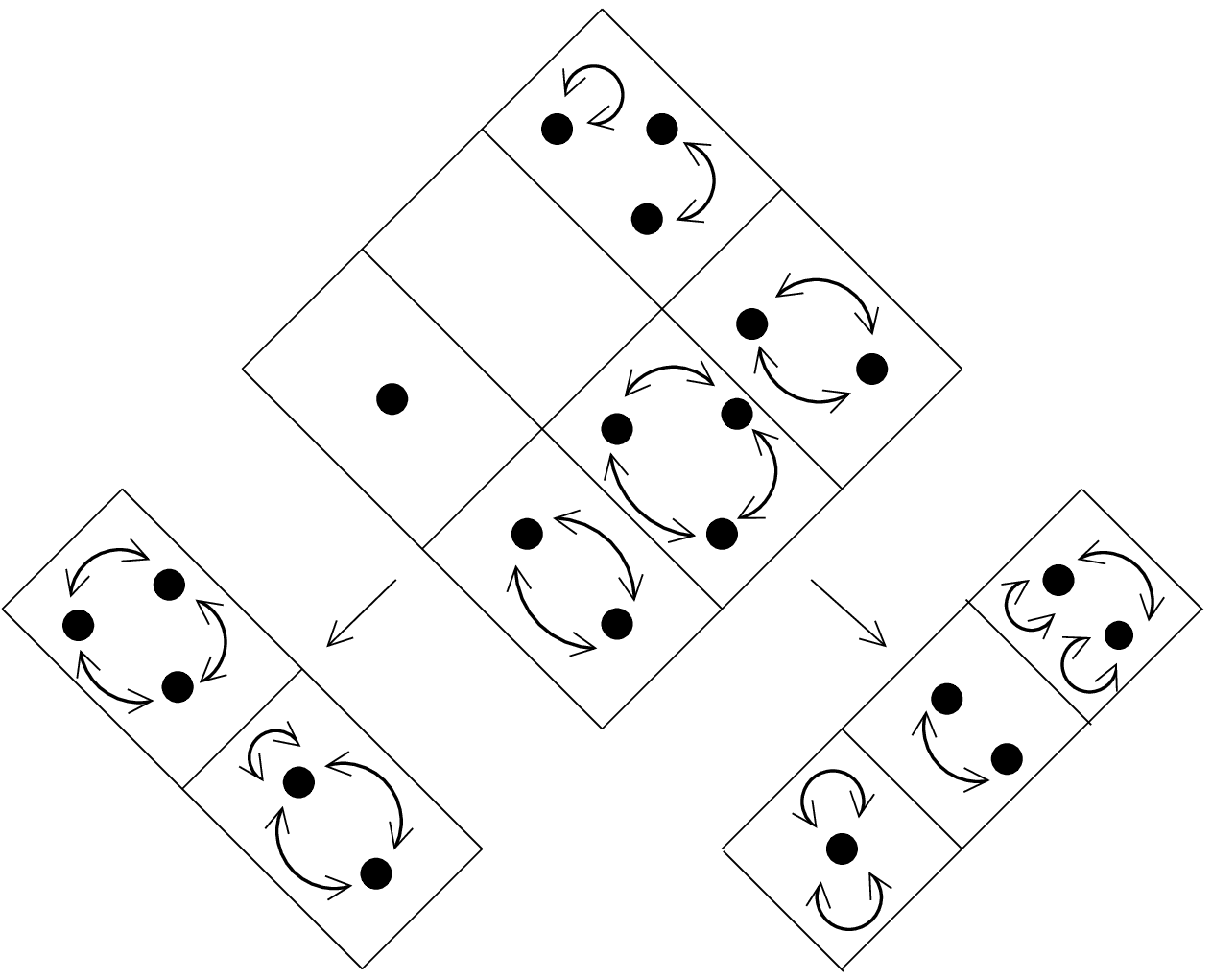}
 \put(-138,88){$q$}
 \put(-52,88){$p$}
 \put(-8,98){$X$}
 \put(-180,98){$Y$}
 \put(-80,170){$S$}
\end{picture}
\]
\noindent
If a span of groupoids is sufficiently nice---our technical term
will be `tame'---we can convert it into a linear operator from
$\R^{\underline X}$ to $\R^{\underline Y}$.  Viewed as a matrix, this
operator will have nonnegative real matrix entries.  So, we have not
succeeded in `groupoidifying' full-fledged quantum mechanics, where
the matrices can be complex.  Still, we have made some progress.

As a sign of this, it turns out that any groupoid $X$ gives not just a
vector space $\R^{\underline X}$, but a real Hilbert space $L^2(X)$.
If $X = E$, the complexification of this Hilbert space is the Hilbert
space of the quantum harmonic oscillator.  The quantum harmonic
oscillator is the simplest system where we can see the usual tools of
quantum field theory at work: for example, Feynman diagrams.  It turns
out that large portions of the theory of Feynman diagrams can be done
with spans of groupoids replacing operators \cite{BaezDolan:2001}.
The combinatorics of these diagrams then becomes vivid, stripped bare
of the trappings of analysis.  We sketch how this works in Section
\ref{feynman}.  A more detailed treatment can be found in the work of
Jeffrey Morton \cite{Morton:2006}.

To get complex numbers into the game, Morton generalizes groupoids to
`groupoids over $\U(1)$': that is, groupoids $X$ equipped with
functors $v \maps X \to \U(1)$, where $\U(1)$ is the
groupoid with unit complex numbers as objects and only identity 
morphisms.  The cardinality of a groupoid over $\U(1)$ can be complex.

Other generalizations of groupoid cardinality are also interesting.
For example, Leinster has generalized it to categories
\cite{Leinster:2008}.  The cardinality of a category can be negative!
More recently, Weinstein has generalized the concept of cardinality 
to Lie groupoids \cite{Weinstein:2008}.  Getting a useful generalization 
of groupoids for which the cardinality is naturally complex, without 
putting in the complex numbers `by hand', remains an elusive goal.  
However, the work of Fiore and Leinster suggests that it is possible
\cite{FioreLeinster:2005}.

In the last few years James Dolan, Todd Trimble and the authors have
applied groupoidification to better understand the process of
`$q$-deformation' \cite{Baez}.  Many important algebraic structures
can be systematically deformed in a way that depends on a parameter
$q$, with $q = 1$ being the `default' case of no deformation at all.
A beautiful story has begun to emerge in which $q$-deformation arises
naturally from replacing the groupoid of pointed finite sets by the
groupoid of finite-dimensional vector spaces over the field with $q$
elements, $\F_q$, where $q$ is a prime power.  We hope to write up
this material and develop it further in the years to come.  This paper
is just a first installment.

For example, any Dynkin diagram with $n$ dots gives rise to a finite
group of linear transformations of $\R^n$ which is generated by
reflections, one for each dot of the Dynkin diagram, which satisfy
some relations, one for each edge.  These groups are called `Coxeter
groups' \cite{Humphreys}.  The simplest example is the symmetry group 
of the regular $n$-simplex, which arises from a Dynkin diagram with
$n$ dots in a row, like this:
\[\xymatrix{ \bullet\ar@{-}[r] & \bullet \ar@{-}[r] & \bullet}\]
This group is generated by $n$ reflections $s_d$, one for each dot
$d$.  These generators obey the Yang--Baxter equation:
\[  s_d s_{d'} s_d = s_{d'} s_d s_{d'} \]
when the dots $d$ and $d'$ are connected by an edge, and they commute
otherwise.  Indeed the symmetry group of the regular $n$-simplex is
just the symmetric group $S_{n+1}$, which acts as permutations of the
vertices, and the generator $s_d$ is the transposition that switches
the $d$th and $(d+1)$st vertices.

Coxeter groups are a rich and fascinating subject, and part of their
charm comes from the fact that the group algebra of any Coxeter group
admits a $q$-deformation, called a `Hecke algebra', which has many of
the properties of the original group (as expressed through its group
algebra).  The Hecke algebra again has one generator for each dot of
the Dynkin diagram, now called $\sigma_d$.  These generators obey
the same relation for each edge that we had in the original Coxeter
group.  The only difference is that while the Coxeter group generators
are reflections, and thus satisfy $s_d^2 = 1$, the Hecke algebra
generators obey a $q$-deformed version of this equation:
\[ \sigma_d^2 = (q-1)\sigma_d + q  . \]

Where do Hecke algebras come from?  They arise in a number of ways,
but one enlightening description involves the theory of `buildings',
where each Dynkin diagram corresponds to a type of geometry
\cite{Brown,Garrett}.  For example, the Dynkin diagram shown above
(with $n$ dots) corresponds to the geometry of projective $n$-space.
Projective $n$-space makes sense over any field, and when $q$ is a
prime power, the Hecke algebra arises in the study of $n$-dimensional 
projective space over the field $\F_q$.  We shall explain this in
detail in the case of the projective plane.  But the $n$-simplex can be 
thought of as a `$q \to 1$ limit' of an $n$-dimensional projective
space over $\F_q$.  The idea is that the vertices, edges, triangles,
and so on of the $n$-simplex play the role of points, lines, planes,
and so on in a degenerate sort of projective space.  In this limiting case,
the Hecke algebra reduces to the group algebra of the symmetric group.  
As it turns out, this fact can be understood more clearly when we 
\textit{groupoidify} the Hecke algebra.  We shall sketch the idea in 
this paper, and give more details in the next paper of this series.  
In the meantime, see \cite{Hoffnung}.

Any Dynkin diagram also gives a geometry over the field $\mathbb{C}$,
and the symmetries of this geometry form a simple Lie group.  The
symmetry transformations close to the identity are described by the
Lie algebra $\g$ of this group---or equally well, by the universal
enveloping algebra $U\g$, which is a Hopf algebra.  This universal
enveloping algebra admits a $q$-deformation, a Hopf algebra $U_q \g$ 
known as a `quantum group'.  There has been a lot of work attempting 
to categorify quantum groups, from the early ideas of Crane and Frenkel 
\cite{CF}, to the work of Khovanov, Lauda and Rouqier \cite{Kh2,KL,Rou2}, 
and beyond.  

Here we sketch how to groupoidify, not the whole quantum group, but only 
its `positive part' $U_q^+ \g$.  When $q = 1$, this positive part is just
the universal enveloping algebra of a chosen maximal nilpotent
subalgebra of $\g$.  The advantage of restricting attention to the
positive part is that $U_q^+ \g$ has a basis in which the formula for
the product involves only \textit{nonnegative} real numbers---and any
such number is the cardinality of some groupoid.

The strategy for groupoidifying $U_q^+ \g$ is implicit in Ringel's
work on Hall algebras \cite{Ringel}.  Suppose we have a `simply-laced'
Dynkin diagram, meaning one where two generators of the Coxeter group
obey the Yang-Baxter equation whenever the corresponding dots are
connected by an edge.  If we pick a direction for each edge of this
Dynkin diagram, we obtain a directed graph.  This in turn freely
generates a category, say $Q$.  The objects in this category are the
dots of the Dynkin diagram, while the morphisms are paths built from
directed edges.

For any prime power $q$, there is a category $\Rep(Q)$ whose
objects are `representations' of $Q$: that is, functors
\[   R \maps Q \to \FinVect_q, \]
where $\FinVect_q$ is the category of finite-dimensional vector spaces
over $\F_q$.  The morphisms in $\Rep(Q)$ are natural transformations.
Thanks to the work of Ringel, one can see that the underlying
groupoid of $\Rep(Q)$---which has only natural \textsl{isomorphisms} 
as morphisms---groupoidifies the vector space $U_q^+ \g$.  Even better, 
we can groupoidify the product in $U_q^+ \g$.  The same sort of 
construction with the category of pointed finite sets replacing 
$\FinVect_q$ lets us handle the $q = 1$ case \cite{Szczesny}.  
So yet again, $q$-deformation is related to the passage from
pointed finite sets to finite-dimensional vector spaces over finite 
fields.

The plan of the paper is as follows.  In Section
\ref{degroupoidification}, we present some basic facts about
`degroupoidification'.  We describe a process that associates to any
groupoid $X$ the vector space $\R^{\u{X}}$ consisting of real-valued
functions on the set of isomorphism classes of objects of $X$, and
associates to any `tame' span of groupoids a linear operator.  In
Section \ref{homology}, we describe a slightly different process,
which associates to $X$ the vector space $\R[\u{X}]$ consisting of
formal linear combinations of isomorphism classes of objects of $X$.
Then we turn to some applications.  Section \ref{feynman} describes
how to groupoidify the theory of Feynman diagrams, Section \ref{hecke}
describes how to groupoidify the theory of Hecke algebras, and Section
\ref{hall} describes how to groupoidify Hall algebras.  In Section
\ref{processes} we prove that degroupoidifying a tame span gives a
well-defined linear operator.  We also give an explicit criterion for
when a span of groupoids is tame, and an explicit formula for the
operator coming from a tame span.  Section \ref{properties} proves
many other results stated earlier in the paper.  Appendix
\ref{appendix} provides some basic definitions and useful lemmas
regarding groupoids and spans of groupoids.  The goal is to make it
easy for readers to try their own hand at groupoidification.

\section{Degroupoidification}\label{degroupoidification}

In this section we describe a systematic process for turning groupoids
into vector spaces and tame spans into linear operators.  This
process, `degroupoidification', is in fact a kind of functor.
`Groupoidification' is the attempt to {\it undo} this functor.  To
`groupoidify' a piece of linear algebra means to take some structure
built from vector spaces and linear operators and try to find
interesting groupoids and spans that degroupoidify to give this
structure.  So, to understand groupoidification, we need to master
degroupoidification.

We begin by describing how to turn a groupoid into a vector space.  In
what follows, all our groupoids will be \textbf{essentially small}.  This
means that they have a {\it set} of isomorphism classes of objects,
not a proper class.  We also assume our groupoids are \textbf{locally
finite}: given any pair of objects, the set of morphisms from one object 
to the other is finite.

\begin{definition}
Given a groupoid $X$, let $\u{X}$ be the set of isomorphism
classes of objects of $X$.
\end{definition}
\begin{definition}
\label{vectorspace}
Given a groupoid $X$, let the {\bf degroupoidification} of $X$ be
the vector space
\[ \R^{\u{X}} = \lbrace \Psi \maps \u{X} \to \R \rbrace. \]
\end{definition}

A nice example is the groupoid of finite sets and bijections:
\begin{example}
\label{power_series}
\textup{
Let $E$ be the groupoid of finite sets and bijections. Then
$\u{E}\iso \N$, so
\[ \R^{\u{E}} \iso \lbrace \psi \maps \N \to \R \rbrace 
\cong \R[[z]] , \]
where the formal power series associated to a function $\psi
\maps \N \to \R$ is given by:
\[  \sum_{n\in\N}\psi(n)z^n. \]
}
\end{example}

A sufficiently nice groupoid over a groupoid $X$ will give a vector
in $\R^{\u{X}}$.  To construct this, we use the concept of
groupoid cardinality:

\begin{definition}
The {\bf cardinality} of a groupoid $X$ is
\[ |X| = \sum_{[x]\in \u{X}} \frac{1}{|\Aut(x)|} \]
where $|\Aut(x)|$ is the cardinality of the automorphism group of an object
$x$ in $X$.  If this sum diverges, we say $|X| = \infty$.
\end{definition}

The cardinality of a groupoid $X$ is a well-defined nonnegative rational
number whenever $\u{X}$ and all the automorphism groups of
objects in $X$ are finite.  More generally, we say:

\begin{definition}
A groupoid $X$ is {\bf tame} if it is essentially small,
locally finite, and $|X| < \infty$.
\end{definition}
\noindent
We show in Lemma \ref{EQUIVGRPD} that given equivalent groupoids
$X$ and $Y$, $|X| = |Y|$.  We give a useful alternative formula
for groupoid cardinality in Lemma \ref{ALTCARD}.

The reason we use $\R$ rather than $\Q$ as our ground field is that
there are interesting groupoids whose cardinalities are irrational numbers.
The following example is fundamental:

\begin{example}
\textup{
The groupoid of finite sets $E$ has cardinality
\[ |E| ~=~ \sum_{n \in \N} \frac{1}{|S_n|} ~=~ 
\sum_{n \in \N} \frac{1}{n!} ~=~ e. \]
}
\end{example}

With the concept of groupoid cardinality in hand, we now describe
how to obtain a vector in $\R^{\u{X}}$ from a sufficiently nice 
groupoid over $X$.  

\begin{definition}
Given a groupoid $X$, a {\bf groupoid over $X$} is a groupoid $\Psi$ 
equipped with a functor $v \maps \Psi \to X$.
\end{definition}
\begin{definition}
Given a groupoid over $X$, say $v \maps \Psi \to X$, and an object $x \in X$,
we define the {\bf full inverse image} of $x$,
denoted $v^{-1}(x)$, to be the groupoid where:
\begin{itemize}
\item
an object is an object $a \in \Psi$ such that $v(a) \cong x$;
\item
a morphism $f \maps a \to a'$ is any morphism in $\Psi$ from $a$ to
$a'$.
\end{itemize}
\end{definition}
\begin{definition}
A groupoid over $X$, say $v \maps \Psi \to X$, is {\bf tame} if the 
groupoid $v^{-1}(x)$ is tame for all $x \in X$. 
\end{definition}

\noindent 
We sometimes loosely say that $\Psi$ is a tame groupoid over $X$.
When we do this, we are referring to a functor $v \maps \Psi \to X$
that is tame in the above sense.  We do not mean that $\Psi$ is tame
as a groupoid.

\begin{definition}
\label{degroupoidification_of_vectors}
Given a tame groupoid over $X$, say $v \maps \Psi\to X$, 
there is a vector $\utilde{\Psi} \in \R^{\u{X}}$ 
defined by:
\[ \utilde{\Psi}([x]) = |v^{-1}(x)|. \]
\end{definition}

\noindent
As discussed in Section \ref{intro}, the theory of generating functions
gives many examples of this construction.  Here is one:

\begin{example}
\label{generating_function}
\textup{
Let $\Psi$ be the groupoid of 2-colored finite sets.  An object of
$\Psi$ is a `2-colored finite set': that is a finite set $S$ equipped
with a function $c: S \to 2$, where $2 = \{0,1\}$.  A morphism of
$\Psi$ is a function between 2-colored finite sets preserving the
2-coloring: that is, a commutative diagram of this sort:
\[
\xymatrix{
& S \ar[dr]_{c} \ar[rr]^{f} && S' \ar[dl]^{c'}  \\
 & & \{0,1\} & & 
}
\]       
There is a forgetful functor $v \maps \Psi \to E$ sending any
2-colored finite set $c \maps S \to 2$ to its underlying set $S$.  It is a
fun exercise to check that for any $n$-element set, say $n$ for short,
the groupoid $v^{-1}(n)$ is equivalent to the weak quotient
$2^n/\!/S_n$, where $2^n$ is the set of functions $c: n \to 2$ and the
permutation group $S_n$ acts on this set in the obvious way.  It follows
that
\[  \utilde{\Psi}(n) = |v^{-1}(n)| = |2^n /\!/ S_n | = 2^n/n! \]
so the corresponding power series is
\[ \utilde{\Psi} = \sum_{n \in \N} \frac{2^n}{n!}z^n = 
e^{2z} \in \R[[z]]. \]
We call this the {\bf generating function} of $v \maps \Psi \to E$,
and indeed it is the usual generating function for 2-colored sets.
Note that the $n!$ in the denominator, often regarded as a convention,
arises naturally from the use of groupoid cardinality.  }
\end{example}

Both addition and scalar multiplication of vectors have groupoidified
analogues.  We can add two groupoids $\Phi$, $\Psi$ over $X$ by taking
their coproduct, i.e., the disjoint union of $\Phi$ and $\Psi$ with
the obvious map to $X$:
\[
\xymatrix{
\Phi + \Psi \ar[d] \\
X
}
\]
We then have:
\begin{proposition*}
Given tame groupoids $\Phi$ and $\Psi$ over $X$,
\[\utilde{\Phi + \Psi} = \utilde{\Phi} + \utilde{\Psi}.\]
\end{proposition*}
\begin{proof}
This will appear later as part of Lemma \ref{addvectors},
which also considers infinite sums.
\end{proof}

We can also multiply a groupoid over $X$ by a `scalar'---that is, a
fixed groupoid.  Given a groupoid over $X$, say $v \maps \Phi\to X$,
and a groupoid $\Lambda$, the cartesian product $\Lambda\times \Psi$
becomes a groupoid over $X$ as follows:
\[\xymatrix{
\Lambda\times \Psi \ar[d]^{v\pi_2}\\ X\\ }\] 
where $\pi_2 \maps \Lambda\times \Psi\to \Psi$ is projection onto the
second factor.  We then have:

\begin{proposition*}
Given a groupoid $\Lambda$ and a groupoid $\Psi$ over $X$, the groupoid
$\Lambda \times \Psi$ over $X$ satisfies
\[\utilde{\Lambda \times \Psi} = |\Lambda|\utilde{\Psi}.\]
\end{proposition*}
\begin{proof}
This is proved as Proposition \ref{scalarmult1}.
\end{proof}

We have seen how degroupoidification turns a groupoid $X$ into a vector space
$\R^{\u{X}}$.  Degroupoidification also turns any sufficiently 
nice span of groupoids into a linear operator.

\begin{definition}
Given groupoids $X$ and $Y$, a {\bf span} from $X$ to $Y$ is a diagram
\[\xymatrix{
 & S\ar[dl]_q\ar[dr]^p & \\
 Y & & X \\
}\]
where $S$ is groupoid and $p\maps S\to X$ and $q \maps S\to Y$ are functors.
\end{definition}

To turn a span of groupoids into a linear operator, we need a
construction called the `weak pullback'.  This construction will let
us apply a span from $X$ to $Y$ to a groupoid over $X$ to obtain a
groupoid over $Y$.  Then, since a tame groupoid over $X$ gives a
vector in $\R^{\u{X}}$, while a tame groupoid over $Y$ gives a
vector in $\R^{\u{Y}}$, a sufficiently nice span from $X$ to
$Y$ will give a map from $\R^{\u{X}}$ to $\R^{\u{Y}}$.
Moreover, this map will be linear.

As a warmup for understanding weak pullbacks for groupoids, we recall
ordinary pullbacks for sets, also called `fibered products'.
The data for constructing such a pullback is a pair of sets equipped
with functions to the same set:
\[
\xymatrix{
& T \ar[dr]_{q} & & S \ar[dl]^{p} & \\
 & & X & & 
}
\]  
The pullback is the set
\[ P = \lbrace (s,t) \in S \times T \, | \; p(s) = q(t) \rbrace  \]
together with the obvious projections $\pi_S \maps P \to S$ and 
$\pi_T \maps P \to T$.  The pullback makes this diamond commute:
\[
\xymatrix{
& & P \ar[dl]_{\pi_T} \ar[dr]^{\pi_S} & &\\
& T \ar[dr]_{q} & & S \ar[dl]^{p} & \\
 & & X & & 
}
\]  
and indeed it is the `universal solution' to the problem of finding
such a commutative diamond \cite{Mac Lane}.

To generalize the pullback to groupoids, we need to weaken one
condition.  The data for constructing a weak pullback is a pair of
groupoids equipped with functors to the same groupoid:
\[
\xymatrix{
& T \ar[dr]_{q} & & S \ar[dl]^{p} & \\
 & & X & & 
}
\]  
But now we replace the {\it equation} in the definition of pullback
by a {\it specified isomorphism}.  So, we define the weak pullback 
$P$ to be the groupoid where an object is a triple $(s,t,\alpha)$ 
consisting of an object $s \in S$, an object $t \in T$, and an 
isomorphism $\alpha \maps p(s) \to q(t)$ in $X$.  A morphism
in $P$ from $(s,t,\alpha)$ to $(s',t',\alpha')$ consists of a morphism
$f \maps s \to s'$ in $S$ and a morphism $g \maps t \to t'$ in $T$
such that the following square commutes:
\[
\xymatrix{
p(s) \ar[d]_{p(f)} \ar[r]^{\alpha} & q(t) \ar[d]^{q(g)} \\
p(s') \ar[r]_{\alpha'} & q(t')
}
\]
Note that any set can be regarded as a {\bf discrete} groupoid:
one with only identity morphisms.  For discrete groupoids, the weak
pullback reduces to the ordinary pullback for sets.
Using the weak pullback, we can apply a span from $X$ to $Y$ to a groupoid
over $X$ and get a groupoid over $Y$.  Given a span of groupoids:
\[
\xymatrix{
& S\ar[dl]_{q} \ar[dr]^{p} & \\
Y & & X
}
\]
and a groupoid over $X$:
\[
\xymatrix{
   & \Psi \ar[dl]_{v} \\
   X &  
}
\]
we can take the weak pullback, which we call $S\Psi$:
\[
\xymatrix{
& & S\Psi \ar[dl]_{\pi_S}\ar[dr]^{\pi_{\Psi}} & \\
& S \ar[dl]_{q} \ar[dr]^{p} & & \Psi \ar[dl]_{v} \\
Y & & X &  
}
\]
and think of $S\Psi$ as a groupoid over $Y$:
\[
\xymatrix{
& S\Psi \ar[dl]_{q \pi_S}  \\
Y  
}
\]
This process will determine a linear operator from $\R^{\u X}$ 
to $\R^{\u Y}$ if the span $S$ is sufficiently nice:
\begin{definition}
A span
\[
\xymatrix{
& S\ar[dl]_{q} \ar[dr]^{p} & \\
Y & & X
}
\]
is {\bf tame} if $v \maps \Psi \to X$ being tame implies that 
$q\pi_{S} \maps S\Psi \to Y$ is tame.
\end{definition}

\begin{theorem*}
Given a tame span:
\[
\xymatrix{
& S\ar[dl]_{q} \ar[dr]^{p} & \\
Y & & X
}
\]
there exists a unique linear operator
\[ \utilde{S} \maps \R^{\u{X}} \to \R^{\u{Y}} \]
such that
\[ \utilde{S}\utilde{\Psi} = \utilde{S\Psi} \]
whenever $\Psi$ is a tame groupoid over $X$.
\end{theorem*}
\begin{proof} This is Theorem \ref{PROCESS1}. \end{proof}

Theorem \ref{matrix1} provides an explicit criterion for when a span
is tame.  This theorem also gives an explicit formula for the the
operator corresponding to a tame span $S$ from $X$ to $Y$.  If
$\u{X}$ and $\u{Y}$ are finite, then
$\R^{\u{X}}$ has a basis given by the isomorphism classes
$[x]$ in $X$, and similarly for $\R^{\u{Y}}$.  With respect to
these bases, the matrix entries of $\utilde{S}$ are given as follows:
\begin{equation}
\label{matrix_entry}
\utilde{S}_{[y][x]} = 
\sum_{[s]\in\u{p^{-1}(x)}\bigcap \u{q^{-1}(y)} }\frac{|\Aut(x)|}{|\Aut(s)|}
\end{equation}
where $|\Aut(x)|$ is the set cardinality of the automorphism group of
$x \in X$, and similarly for $|\Aut(s)|$.  Even when $\u{X}$
and $\u{Y}$ are not finite, we have the following formula for
$\utilde S$ applied to $\psi \in \R^{\u X}$:
\begin{equation}
\label{operator_formula}
(\utilde{S} \psi)([y]) = 
\sum_{[x] \in \u{X}} \;\,
\sum_{[s]\in\u{p^{-1}(x)}\bigcap \u{q^{-1}(y)}}
\frac{|\Aut(x)|}{|\Aut(s)|} \,\, \psi([x]) \, .
\end{equation}

As with vectors, there are groupoidified analogues of addition and scalar
multiplication for operators.  Given two spans from $X$ to $Y$:
\[
\xymatrix{
& S\ar[dl]_{q_S} \ar[dr]^{p_S} & & & T\ar[dl]_{q_T} \ar[dr]^{p_T} & \\
Y & & X & Y & & X
}
\]
we can add them as follows.  By the universal property of the
coproduct we obtain from the right legs of the above spans a functor
from the disjoint union $S + T$ to $X$.  Similarly, from the left legs
of the above spans, we obtain a functor from $S + T$ to $Y$.  Thus, we
obtain a span
\[
\xymatrix{
& S + T\ar[dl] \ar[dr] & \\
Y & & X
}
\]
This addition of spans is compatible with degroupoidification:
\begin{proposition*}
If $S$ and $T$ are tame spans from $X$ to $Y$, then so is $S + T$, and
\[    \utilde{S+T} = \utilde{S} + \utilde{T} .\]
\end{proposition*}
\begin{proof} This is proved as Proposition \ref{add_spans}.
\end{proof}

We can also multiply a span by a `scalar': that is, a fixed groupoid.
Given a groupoid $\Lambda$ and a span
\[
\xymatrix{
& S\ar[dl]_q \ar[dr]^p & \\
Y & & X
}
\]
we can multiply them to obtain a span
\[
\xymatrix{
& \Lambda \times S \ar[dl]_{q\pi_2} \ar[dr]^{p\pi_2} & \\
Y & & X
}
\]
Again, we have compatibility with degroupoidification:

\begin{proposition*}
Given a tame groupoid $\Lambda$ and a tame span
\[
\xymatrix{
& S\ar[dl] \ar[dr] & \\
Y & & X
}
\]
then $\Lambda \times S$ is tame and
\[\utilde{\Lambda \times S} = |\Lambda| \, \utilde{S}.\]
\end{proposition*}
\begin{proof} This is proved as Proposition \ref{scalarmult2}.
\end{proof}

Next we turn to the all-important process of {\it composing} spans.
This is the groupoidified analogue of matrix multiplication.  Suppose
we have a span from $X$ to $Y$ and a span from $Y$ to $Z$:
\[
\xymatrix{
& T\ar[dl]_{q_T} \ar[dr]^{p_T} & & S \ar[dl]_{q_S} \ar[dr]^{p_S}& \\
Z & & Y & & X
}
\]  
Then we say these spans are {\bf composable}.
In this case we can form a weak pullback in the middle:
\[
\xymatrix{
& & TS \ar[dl]_{\pi_T} \ar[dr]^{\pi_S} & &\\
& T\ar[dl]_{q_T} \ar[dr]^{p_T} & & S \ar[dl]_{q_S} \ar[dr]^{p_S}& \\
Z & & Y & & X
}
\]  
which gives a span from $X$ to $Z$:
\[\xymatrix{
 & TS \ar[dl]_{q_T \pi_T} \ar[dr]^{p_S\pi_S} & \\
 Z & & X \\
}\]
called the {\bf composite} $TS$.

When all the groupoids involved are discrete, the spans $S$ and $T$
are just matrices of sets, as explained in Section \ref{intro}.  We urge
the reader to check that in this case, the process of composing spans
is really just matrix multiplication, with cartesian product of sets
taking the place of multiplication of numbers, and disjoint union of
sets taking the place of addition:
\[    (TS)^k_j = \coprod_{j \in Y} T^k_j \times S^j_i . \]
So, composing spans of groupoids is a generalization of matrix
multiplication, with weak pullback playing the role of summing
over the repeated index $j$ in the formula above.

So, it should not be surprising that degroupoidification sends
a composite of tame spans to the composite of their corresponding
operators:

\begin{proposition*} If $S$ and $T$ are composable tame spans:
\[
\xymatrix{
& T\ar[dl]_{q_T} \ar[dr]^{p_T} & & S \ar[dl]_{q_S} \ar[dr]^{p_S}& \\
Z & & Y & & X
}
\]  
then the composite span
\[\xymatrix{
 & TS \ar[dl]_{q_T \pi_T} \ar[dr]^{p_S\pi_S} & \\
 Z & & X \\
}\]
is also tame, and
\[       \utilde{TS} = \utilde{T} \utilde{S}  .\]
\end{proposition*}

\begin{proof} This is proved as Lemma \ref{composition}.
\end{proof}

Besides addition and scalar multiplication, there is an extra
operation for groupoids over a groupoid $X$, which is the reason
groupoidification is connected to quantum mechanics.  Namely, we can
take their inner product:

\begin{definition}
Given groupoids $\Phi$ and $\Psi$ over $X$, we define
the {\bf inner product} $\ip{\Phi}{\Psi}$ to be this weak pullback:
\[
\xymatrix{
& \ip{\Phi}{\Psi} \ar[dl] \ar[dr] & \\
\Phi \ar[dr] & & \Psi \ar[dl] \\
& X &
}
\]
\end{definition}
\begin{definition}
\label{L2}
A groupoid $\Psi$ over $X$ is called {\bf square-integrable} if
$\ip{\Psi}{\Psi}$ is tame.  We define $L^2(X)$ to be the subspace of
$\R^{\u{X}}$ consisting of finite real linear combinations of vectors
$\utilde{\Psi}$ where $\Psi$ is square-integrable.
\end{definition}

Note that $L^2(X)$ is all of $\R^{\u{X}}$ when $\u{X}$
is finite.  The inner product of groupoids over $X$ makes $L^2(X)$
into a real Hilbert space:
\begin{theorem*}
Given a groupoid $X$, there is a unique inner product $\ip{\cdot}{\cdot}$ 
on the vector space $L^2(X)$ such that
\[ \ip{\utilde{\Phi}}{\utilde{\Psi}} = |\ip{\Phi}{\Psi}| \]
whenever $\Phi$ and $\Psi$ are square-integrable groupoids over $X$.
With this inner product $L^2(X)$ is a real Hilbert space.
\end{theorem*}
\begin{proof} This is proven later as Theorem \ref{innerprod_theorem}. 
\end{proof}
We can always complexify $L^2(X)$ and obtain a complex
Hilbert space.  We work with real coefficients simply to admit that
groupoidification as described here does not make essential use of the
complex numbers.  Morton's generalization involving groupoids over
$\U(1)$ is one way to address this issue \cite{Morton:2006}.

The inner product of groupoids over $X$ has the properties one 
would expect:
\begin{proposition*}
Given a groupoid $\Lambda$ and square-integrable
groupoids $\Phi$, $\Psi$, and $\Psi'$ over $X$, we have the
following equivalences of groupoids:
\begin{enumerate}
\item
\[\ip{\Phi}{\Psi} \simeq \ip{\Psi}{\Phi}.\] 
\item
\[\ip{\Phi}{\Psi + \Psi'} \simeq \ip{\Phi}{\Psi} + \ip{\Phi}{\Psi'}.\] 
\item
\[\ip{\Phi}{\Lambda \times \Psi} \simeq \Lambda \times \ip{\Phi}{\Psi}.\]
\end{enumerate}
\end{proposition*}
\begin{proof} Here equivalence of groupoids is defined in the usual
way---see Definition \ref{equivalence_of_groupoids}.  This result is
proved below as Proposition
\ref{innerprodandadjointprops}. \end{proof}

Just as we can define the adjoint of an operator between
Hilbert spaces, we can define the adjoint of a span of groupoids:

\begin{definition}
Given a span of groupoids from $X$ to $Y$:
\[
\xymatrix{
& S\ar[dl]_{q} \ar[dr]^{p} & \\
Y & & X
}
\]
its {\bf adjoint} $S^{\dagger}$ is the following span of
groupoids from $Y$ to $X$:
\[
\xymatrix{
& S\ar[dl]_{p} \ar[dr]^{q} & \\
X & & Y
}
\]
\end{definition}

We warn the reader that the adjoint of a tame span may not be tame,
due to an asymmetry in the criterion for tameness, Theorem
\ref{matrix1}.  But of course a span of finite groupoids is tame, and
so is its adjoint.  Moreover, we have:

\begin{proposition*}
Given a span
\[
\xymatrix{
& S \ar[dl]_{q} \ar[dr]^{p} & \\
Y & & X
}
\]
and a pair $v \maps \Psi \to X$, $w \maps \Phi \to Y$ of groupoids
over $X$ and $Y$, respectively, there is an equivalence of groupoids
\[ \langle \Phi,S\Psi \rangle 
\simeq \langle S^\dagger\Phi,\Psi \rangle. \]
\end{proposition*}

\begin{proof} 
This is proven as Proposition \ref{innerprod_adjoint}.
\end{proof}

We say what it means for spans to be `equivalent' in Definition
\ref{EQUIVALENCE}.  Equivalent tame spans give the same linear
operator: $S \simeq T$ implies $\utilde{S} = \utilde{T}$.  Spans of
groupoids obey many of the basic laws of linear algebra---up to
equivalence.  For example, we have these familiar properties of
adjoints:

\begin{proposition*}
Given spans
\[
\xymatrix{
& T \ar[dl]_{q_T} \ar[dr]^{p_T} & & & S \ar[dl]_{q_S} \ar[dr]^{p_S} & \\
Z & & Y & Y & & X
}
\]
and a groupoid $\Lambda$, we have the following equivalences of 
spans:
\begin{enumerate}
\item $(TS)^\dagger \simeq S^\dagger T^\dagger$
\item $(S+T)^\dagger \simeq S^\dagger + T^\dagger$
\item $(\Lambda S)^\dagger \simeq \Lambda S^\dagger$
\end{enumerate}
\end{proposition*}

\begin{proof}
These will follow easily after we show addition and composition 
of spans and scalar multiplication are well defined.
\end{proof}

\noindent
In fact, degroupoidification is a functor 
\[         \utilde{\;\;\;} \,\maps \Span \to \Vect   \]
where $\Vect$ is the category of real vector spaces and linear
operators, and $\Span$ is a category with 
\begin{itemize}
\item
groupoids as objects,
\item
equivalence classes of tame spans as morphisms,
\end{itemize}
where composition comes from the method of composing spans we have
just described.  We prove this fact in Theorem \ref{functor}.  A
deeper approach, which we shall explain elsewhere, is to think of
$\Span$ as a weak 2-category (i.e., bicategory) with:
\begin{itemize}
\item
groupoids as objects,
\item
tame spans as morphisms,
\item
isomorphism classes of maps of spans as 2-morphisms
\end{itemize}
Then degroupoidification becomes a functor between weak 2-categories:
\[         \utilde{\;\;} \, \maps \Span \to \Vect   \]
where $\Vect$ is viewed as a weak 2-category with only identity
2-morphisms.  So, groupoidification is not merely a way of replacing
linear algebraic structures involving the real numbers with purely
combinatorial structures.  It is also a form of `categorification'
\cite{BaezDolan:1998}, where we take structures defined in the
category $\Vect$ and find analogues that live in the weak 2-category
$\Span$.

We could go even further and think of $\Span$ as a weak 3-category with 
\begin{itemize}
\item
groupoids as objects,
\item
tame spans as morphisms,
\item
maps of spans as 2-morphisms,
\item
maps of maps of spans as 3-morphisms.
\end{itemize}
However, we have not yet found a use for this further structure.

Lastly we would like to say a few words about tensors and traces. 
We can define the {\bf tensor product} of groupoids $X$ and $Y$ 
to be their cartesian product $X \times Y$, and the {\bf tensor 
product} of spans
\[
\xymatrix{
& S \ar[dl]_{q} \ar[dr]^{p} & & & S' \ar[dl]_{q'} \ar[dr]^{p'} & \\
Y & & X & Y' & & X'
}
\]
to be the span
\[
\xy
(0,15)*{S\times S'}="B";(-15,0)*{Y\times Y'}="A";(15,0)*{X\times X'}="C";
{\ar_{q \times q'} (-2,13);(-13,2)};
{\ar^{p \times p'} (2,13);(13,2)};
\endxy
\]
Defining the tensor product of maps of spans in a similar way, we
conjecture that $\Span$ actually becomes a \textit{symmetric monoidal}
weak 2-category\cite{McCrudden}.  If this is true, then
degroupoidification should be a `lax symmetric monoidal functor',
thanks to the natural map
\[       \R^{\u{X}} \tensor \R^{\u{Y}} \to \R^{\u{X \times Y}} \, .\]
The word `lax' refers to the fact that this map is not an isomorphism
of vector spaces unless either $X$ or $Y$ has finitely many isomorphism
classes of objects.  In the next section we present an alternative
approach to degroupoidification that avoids this problem.  The idea
is simple: instead of working with the vector space $\R^{\u{X}}$ consisting
of all functions on $\u{X}$, we work with the vector space $\R[\u{X}]$
having $\u{X}$ as its basis.  Then we have
\[       \R[\u{X}] \tensor \R[\u{Y}] \iso \R[\u{X \times Y}]  .\]
In fact both approaches to groupoidification have their own advantages,
and they are closely related, since
\[       \R^{\u{X}} \iso \R[\u{X}]^*  \, .\]

Regardless of these nuances, the important thing about the `monoidal'
aspect of degroupoidification is that it lets us mimic all the usual
manipulations for tensors with groupoids replacing vector spaces.  
Physicists call a linear map
\[          S \maps V_1 \tensor \cdots \tensor V_m \to 
                    W_1 \tensor \cdots \tensor W_n  \]
a \textbf{tensor}, and denote it by an expression
\[              S^{j_1 \cdots j_n}_{i_1 \cdots i_m}  \]
with one subscript for each `input' vector space $V_1, \dots, V_m$,
and one superscript for each `output' vector space $W_1, \dots, W_n$.
In the traditional approach to tensors, these indices label bases of
the vector spaces in question.  Then the expression $S^{j_1 \cdots
j_n}_{i_1 \cdots i_m}$ stands for an array of numbers: the components
of the tensor $S$ with respect to the chosen bases.  This lets us
describe various operations on tensors by multiplying such expressions
and summing over indices that appear repeatedly, once as a superscript
and once as a subscript.  In the more modern `string diagram'
approach, these indices are simply names of input and output wires for
a black box labelled $S$:
\begin{center}
\begin{picture}(60,115)(0,-10)
 \includegraphics[scale=0.35]{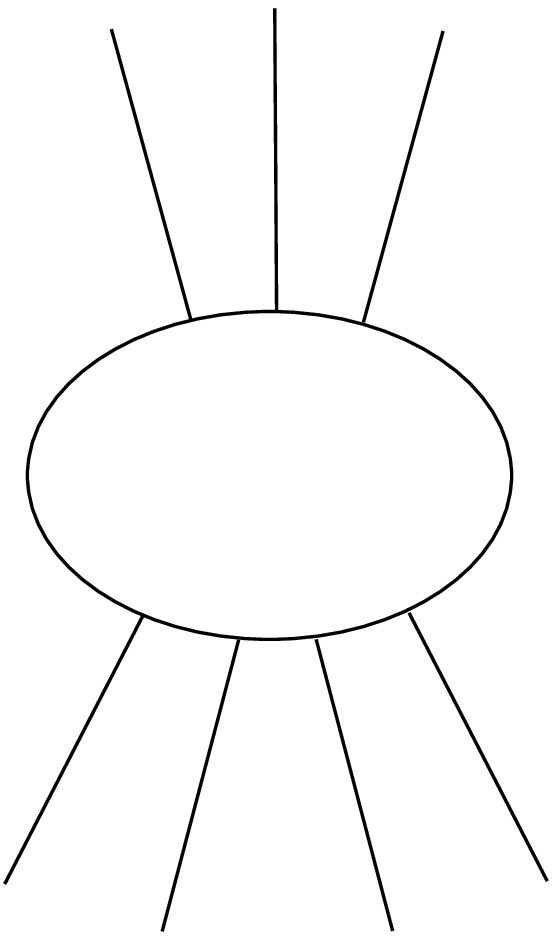}
 \put(-62,-5){$i_1$}
 \put(-43,-9){$i_2$}
 \put(-20,-9){$i_3$}
 \put(-1,-5){$i_4$}
 \put(-51,96){$j_1$}
 \put(-33,98){$j_2$}
 \put(-14,96){$j_3$}
 \put(-32,44){$S$}
\end{picture}
\end{center}
Here we are following physics conventions, where inputs are at
the bottom and outputs are at the top.  In this approach, when
an index appears once as a superscript and once as a subscript,
it means we attach an output wire of one black box to an input of another.

The most famous example is matrix multiplication:
\[   (TS)^k_i = T_j^k S_i^j. \]
Here is the corresponding string diagram:
\begin{center}
\begin{picture}(60,110)(-20,0)
\includegraphics[scale=0.3]{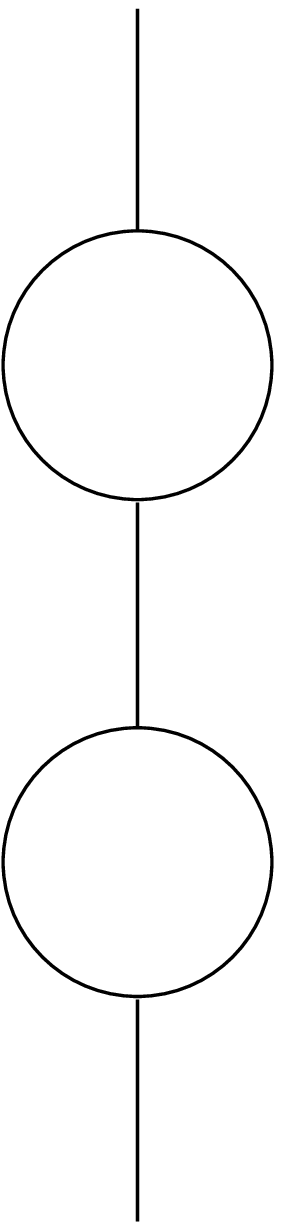}
 \put(-20,94){$k$}
 \put(-20,50){$j$}
 \put(-20,8){$i$}
 \put(-16,71){$T$}
 \put(-16,28){$S$}  
\end{picture}
\end{center}
Another famous example is the trace of a linear operator
$S \maps V \to V$, which is the sum of its diagonal entries:
\[            \tr(S) = S^i_i  \]
As a string diagram, this looks like:
\begin{center}
\begin{picture}(60,120)(12,0)
\includegraphics[scale=0.4]{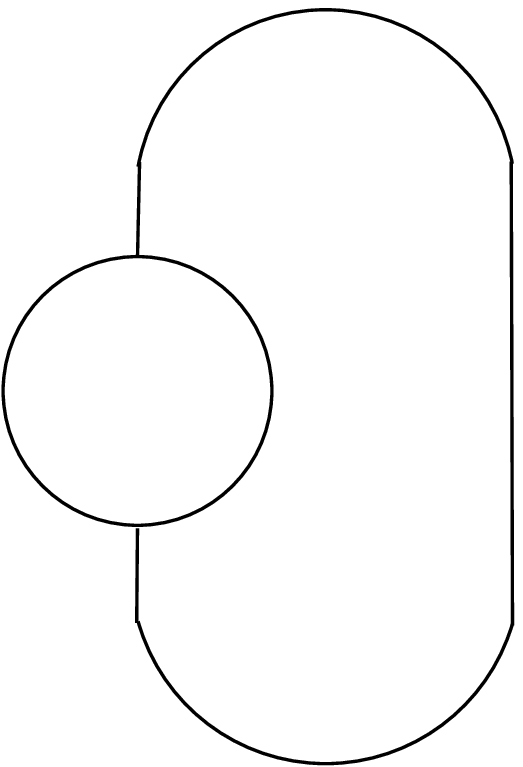} 
 \put(-51,66){$i$}
 \put(-51,16){$i$}
 \put(-49,40){$S$}
\end{picture}
\end{center}
Here the sum is only guaranteed to converge if $V$ is finite-dimensional,
and indeed the full collection of tensor operations is defined only for
finite-dimensional vector spaces.

All these ideas work just as well with spans of groupoids
\[
\xy
(0,15)*{S}="B";(-15,0)*{Y_1 \times \cdots \times Y_n}="A";(15,0)*{X_1 \times \cdots \times X_m}="C";
{\ar_{q} (-2,13);(-13,2)};
{\ar^{p} (2,13);(13,2)};
\endxy
\]
\noindent
taking the place of tensors.  The idea is that \textit{weak pullback 
takes the place of summation over repeated indices}.  Even better, there 
is no need to impose any finiteness or tameness conditions until we 
degroupoidify.  

We have already seen the simplest example: composition of spans via
weak pullback is a generalization of matrix multiplication.
For a trickier one, emphasized by Urs Schreiber \cite{Urs}, consider 
the trace of a span:
\[
\xymatrix{
& S \ar[dl]_q \ar[dr]^p & \\
X & & X  
}
\]
Here it is a bit hard to see which weak pullback to do!
We can get around this problem using an alternate formula for
the trace of a linear map $S \maps V \to V$:
\begin{equation}
\label{trace}
            \tr(S) = g_{jk} S^j_i g^{ik} 
\end{equation}
Here $g_{jk}$ is the tensor corresponding to an arbitrary inner product
$g \maps V \tensor V \to \R$.  In the finite-dimensional case, any
such inner product determines an isomorphism $V \cong V^*$, so we
can interpret the adjoint of $g$ as a linear map $\tilde{g} \maps \R \to 
V \tensor V$, and the tensor for this is customarily written as 
$g$ with superscripts: $g^{ik}$.  Equation
\ref{trace} says that the operator
\[ 
       \R \stackrel{\tilde{g}}{\longrightarrow}
        V \tensor V \stackrel{S \otimes 1}{\longrightarrow}
        V \tensor V \stackrel{g}{\longrightarrow}
       \R
\]
is multiplication by $\tr(S)$.   We can draw $g$ 
as a `cup' and $\tilde{g}$ as a `cap', giving this string diagram:
\begin{center}
\begin{picture}(60,120)(12,0)
\includegraphics[scale=0.4]{trace.eps} 
 \put(-0,66){$k$}
 \put(-0,17){$k$}
 \put(-51,66){$j$}
 \put(-51,17){$i$}
 \put(-49,40){$S$}
\end{picture}
\end{center}

Now let us see how to implement this formula for the trace at the
groupoidified level, to define the trace of a span of groupoids.
Any groupoid $X$ automatically comes equipped with a span
\[
\xy
(0,15)*{\emph{ }X\emph{ }}="B";(-15,0)*{1}="A";(15,0)*{X\times X}="C";
{\ar_{} (-2,13);(-13,2)};
{\ar^{\Delta} (2,13);(13,2)};
\endxy
\]
where $\Delta$ is the diagonal map
and the left-hand arrow is the unique functor to the \textbf{terminal 
groupoid}---that is, the groupoid $1$ with one object and one morphism.  
We can check that at least when $\u{X}$ is finite, degroupoidifying
this span gives an operator
\[
\begin{array}{cccl}
g \maps & \R^{\u{X}} \tensor \R^{\u{X}} & \to & \R  \\
        &  \phi \tensor \psi        & \mapsto & \langle \phi, \psi \rangle 
\end{array}
\]
corresponding to the already described inner product on $\R^{\u{X}}$.  
Similarly, the span
\[\xy
(0,15)*{\emph{ }X\emph{ }}="B";(-15,0)*{X\times X}="A";(15,0)*{1}="C";
{\ar_{} (-2,13);(-13,2)};
{\ar^{\Delta} (2,13);(13,2)};
\endxy\]
degroupoidifies to give the operator
\[  \tilde{g} \maps \R  \to \R^{\u{X}} \tensor \R^{\u{X}} \, .\]
So, to implement Equation \ref{trace} at the level of groupoids 
and define the trace of this span:
\[
\xymatrix{
& S \ar[dl]_q \ar[dr]^p & \\
X & & X  
}
\]
we should take the composite of these three spans:
\[\xy
(0,15)*{S \times X}="B";(-15,0)*{X\times X}="A";(15,0)*{X\times X}="C";
{\ar_{q\times 1} (-2,13);(-13,2)};
{\ar^{p\times 1} (2,13);(13,2)};
(-30,15)*{X}="D";(-45,0)*{1}="E";
{\ar_{} (-32,13);(-43,2)};
{\ar^{\Delta} (-28,13);(-17,2)};
(30,15)*{X}="F";(45,0)*{1}="G";
{\ar_{\Delta} (28,13);(17,2)};
{\ar^{} (32,13);(43,2)};
\endxy
\]
The result is a span from $1$ to $1$, whose apex is a groupoid we
define to be the {\bf trace} $\tr(S)$.  We leave it as an
exercise to check the following basic properties of the trace:

\begin{proposition*}
Given a span of groupoids
\[
\xymatrix{
& S \ar[dl]_q \ar[dr]^p & \\
X & & X  
}
\]
its trace $\tr(S)$ is equivalent to the groupoid for which:
\begin{itemize}
\item an object is a pair $(s,\alpha)$ consisting of
an object $s \in S$ and a morphism $\alpha \maps p(s) \to q(s)$;
\item a morphism from $(s,\alpha)$ to $(s',\alpha')$ is a morphism
$f \maps s \to s'$ such that
\[
\xymatrix{
p(s) \ar[d]_{p(f)} \ar[r]^{\alpha} & q(s) \ar[d]^{q(f)} \\
p(s') \ar[r]_{\alpha'} & q(s')
}
\]
commutes.
\end{itemize}
\end{proposition*}

\begin{proposition*}
Given a span of groupoids
\[
\xymatrix{
& S \ar[dl]_q \ar[dr]^p & \\
X & & X  
}
\]
where $\u{X}$ is finite, we have
\[           |\tr(S)| = \tr({\utilde{S}}) .\]
\end{proposition*}

\begin{proposition*}
Given spans of groupoids
\[
\xymatrix{
& S \ar[dl]_{q_S} \ar[dr]^{p_S} & & & T \ar[dl]_{q_T} \ar[dr]^{p_T} & \\
X & & X & X & & X
}
\]
and a groupoid $\Lambda$, we have the following equivalences of 
groupoids:
\begin{enumerate}
\item $\tr(S + T) \simeq \tr(S) + \tr(T)$  
\item $\tr(\Lambda \times S) \simeq \Lambda \times \tr(S)$ 
\end{enumerate}
\end{proposition*}

\begin{proposition*}
Given spans of groupoids 
\[
\xymatrix{
& T \ar[dl]_{q_T} \ar[dr]^{p_T} & & & S \ar[dl]_{q_S} \ar[dr]^{p_S} & \\
X & & Y & Y & & X
}
\]
we have an equivalence of groupoids
\[           \tr(ST) \simeq \tr(TS) .\]
\end{proposition*}

We could go even further generalizing ideas from vector spaces and
linear operators to groupoids and spans, but at this point the reader
is probably hungry for some concrete applications.  For these, proceed
directly to Section \ref{applications}.  Section \ref{homology} can be
skipped on first reading, since we need it only for the application to
Hall algebras in Section \ref{hall}.

\section{Homology versus Cohomology} \label{homology}

The work we have described so far has its roots in the cohomology of
groupoids.  Any groupoid $X$ can be turned into a topological space,
namely the geometric realization of its nerve \cite{GJ}, and we can
define the cohomology of $X$ to be the cohomology of this space.  The
set of connected components of this space is just the set of
isomorphism classes of $X$, which we have denoted $\u{X}$.
So, the {\bf zeroth cohomology} of the groupoid $X$, with real
coefficients, is precisely the vector space $\R^{\u{X}}$ that
we have been calling the degroupoidification of $X$.

Indeed, one reason degroupoidification has been overlooked until
recently is that every groupoid is equivalent to a disjoint union of
one-object groupoids, which we may think of as groups.  To turn a
groupoid into a topological space it suffices to do this for each of
these groups and then take the disjoint union.  The space associated
to a group $G$ is quite famous: it is the Eilenberg--Mac Lane space
$K(G,1)$.  Similarly, the cohomology of groups is a famous and
well-studied subject.  But the zeroth cohomology of a group is always
just $\R$.  So, zeroth cohomology is not considered interesting.
Zeroth cohomology only becomes interesting when we move from groups to
groupoids---and then only when we consider how a tame span of
groupoids induces a map on zeroth cohomology.

These reflections suggest an alternate approach to degroupoidification
based on homology instead of cohomology:

\begin{definition} Given a groupoid $X$, let the {\bf zeroth homology} of 
$X$ be the real vector space with the set $\u{X}$ as basis.  
We denote this vector space as $\R[\u{X}]$.
\end{definition}

We can also think of $\R[\u{X}]$ as the space of real-valued
functions on $\u{X}$ with finite support.  This makes it clear
that the zeroth homology $\R[\u{X}]$ can be identified with a
subspace of the zeroth cohomology $\R^{\u{X}}$.  If $X$ has
finitely many isomorphism classes of objects, then the set
$\u{X}$ is finite, and the zeroth homology and zeroth
cohomology are canonically isomorphic.  For a groupoid with infinitely
many isomorphism classes, however, the difference becomes important.
The following example makes this clear:
\begin{example}
\label{polynomials}
\textup{
Let $E$ be the groupoid of finite sets and bijections. 
In Example \ref{power_series} we saw that
\[ \R^{\u{E}} \; \iso \; \lbrace \psi \maps \N \to \R \rbrace 
\; \cong \; \R[[z]] . \]
So, elements of $\R[\u{E}]$ may be identified with 
formal power series with only finitely many nonzero coefficients.
But these are just polynomials:
\[   \R[\u{E}] \; \cong \; \R[z]  .\]
}
\end{example}

Before pursuing a version of degroupoidification based on homology, we
should ask if there are other choices built into our recipe for
degroupoidification that we can change. The answer is yes.  Recalling
Equation \ref{operator_formula}, which describes the operator associated
to a tame span:
\[   
(\utilde{S} \psi)([y]) = 
\sum_{[x] \in \u{X}} \;\,
\sum_{[s]\in\u{p^{-1}(x)}\bigcap \u{q^{-1}(y)}}
\frac{|\Aut(x)|}{|\Aut(s)|} \,\, \psi([x]) \,. 
\]
one might wonder about the asymmetry of this formula.  Specifically,
one might wonder why this formula uses information about $\Aut(x)$ but
not $\Aut(y)$. The answer is that we made an arbitrary choice of
conventions. There is another equally nice choice, and in fact 
an entire family of choices interpolating between these two:

\begin{proposition}
\label{cohomological_alpha_degroupoidification}
Given $\alpha \in \R$ and a tame span of groupoids:
\[\xymatrix{
 & S\ar[dl]_q\ar[dr]^p & \\
 Y & & X \\
}\]
there is a linear operator 
called its {\bf \boldmath $\alpha$-degroupoidification}:
\[  \utilde{S_\alpha}\maps \R^{\u{X}}\to\R^{\u{Y}} \]
given by:
\[   
(\utilde{S_\alpha} \psi)([y]) = 
\sum_{[x] \in \u{X}} \;\,
\sum_{[s]\in\u{p^{-1}(x)}\bigcap \u{q^{-1}(y)}}
\frac{|\Aut(x)|^{1-\alpha}|\Aut(y)|^\alpha}{|\Aut(s)|} \,\, \psi([x]) \,. 
\]
\end{proposition}

\begin{proof} 
The only thing that needs to be checked is that the sums converge
for any fixed choice of $y \in Y$.  This follows from our explicit
criterion for tameness of spans, Theorem \ref{matrix1}. 
\end{proof}

The most interesting choices of $\alpha$ are $\alpha = 0$, $\alpha =
1$, and the symmetrical choice $\alpha = 1/2$.  The last convention
has the advantage that for a tame span $S$ with tame adjoint
$S^\dagger$, the matrix for the degroupoidification of $S^\dagger$ is
just the transpose of that for $S$.  We can show:

\begin{proposition}\label{alphaspancompose}
For any $\alpha \in \R$ there is a functor from the category of
groupoids and equivalence classes of tame spans to the category of
real vector spaces and linear
operators, sending:
\begin{itemize}
\item
each groupoid $X$ to its zeroth cohomology $\R^{\u{X}}\,$, and
\item 
each tame span $S$ from $X$ to $Y$ to the operator
$\utilde{S_\alpha} \maps \R^{\u{X}} \to \R^{\u{Y}}\,$.
\end{itemize}
In particular, if we have composable tame spans:
\[
\xymatrix{
& T\ar[dl]_{} \ar[dr]^{} & & S \ar[dl]_{} \ar[dr]^{}& \\
Z & & Y & & X
}
\]  
then their composite
\[\xymatrix{
 & TS \ar[dl]_{} \ar[dr]^{} & \\
 Z & & X \\
}\]
is again tame, and
\[       \utilde{(TS)_\alpha} = \utilde{T_\alpha} \utilde{S_\alpha}  .\]
\end{proposition}

We omit the proof because it mimics that of Theorem \ref{functor}.
Note that a groupoid $\Psi$ over $X$ can be seen as a special case of
a span, namely a span
\[\xymatrix{
 & \Psi \ar[dl]_{v} \ar[dr]^{} & \\
 X & & 1 \\
}\]
where $1$ is the terminal groupoid---that is, the groupoid
with one object and one morphism.  So, $\alpha$-degroupoidification
also gives a recipe for turning $\Psi$ into a vector in $\R^{\u{X}}\,$:
\begin{equation}
\label{alpha_degroupoidification_formula}
\utilde{\Psi_\alpha}{[x]} = |\Aut(x)|^\alpha {|v^{-1}(x)|} \,.
\end{equation}
This idea yields the following result as a special case of Proposition
\ref{alphaspancompose}:

\begin{proposition} 
\label{alpha_degroupoidification_vector}
Given a tame span:
\[\xymatrix{
& S \ar[dl]_{q} \ar[dr]^{p}& \\
Y & & X
}\]
and a tame groupoid over $X$, say $v\maps\Psi\to X$, then
$\utilde{(S\Psi)_\alpha}=\utilde{S_\alpha}\utilde{\Psi_\alpha}$
\end{proposition}

Equation \ref{alpha_degroupoidification_formula} also implies that we 
can compensate for a different choice of $\alpha$ by doing a change of
basis.  So, our choice of $\alpha$ is merely a matter of convenience.
More precisely:

\begin{proposition} 
Regardless of the value of $\alpha \in \R$, the
functors in Proposition \ref{alphaspancompose} are
all naturally isomorphic.
\end{proposition}

\begin{proof} 
Given $\alpha,\beta \in \R$, Equation 
\ref{alpha_degroupoidification_formula} implies that
\[
\utilde{\Psi_\alpha}{[x]} = |\Aut(x)|^{\alpha-\beta} 
\utilde{\Psi_\beta}{[x]} 
\]
for any tame groupoid $\Psi$ over a groupoid $X$.  So, for any
groupoid $X$, define a linear operator
\[    T_X \maps \R[\u{X}] \to \R[\u{X}]  \]
by
\[    (T_X \psi)([x]) = |\Aut(x)|^{\alpha-\beta} \psi([x])  .\]
We thus have
\[
\utilde{\Psi_\alpha} = T_X \utilde{\Psi_\beta}  .
\]

By Proposition
\ref{alpha_degroupoidification_vector}, for any tame span $S$ 
from $X$ to $Y$ and any tame groupoid $\Psi$ over $X$ we have
\[
\begin{array}{ccl}
 T_Y \utilde{S_\alpha} \utilde{\Psi_\alpha} 
&=& T_Y \utilde{(S\Psi)_\alpha} \\
&=& \utilde{(S\Psi)_\beta} \\
&=& \utilde{S_\beta} \utilde{\Psi_\beta} \\
&=& \utilde{S_\beta} T_X \utilde{\Psi_\alpha} 
\end{array}
\]
Since this is true for all such $\Psi$, this implies
\[ T_Y \utilde{S_\alpha} = \utilde{S_\beta} T_X  .\]
So, $T$ defines a natural isomorphism between $\alpha$-degroupoidification
and $\beta$-degroupoidification.
\end{proof}

Now let us return to homology.  We can also do
$\alpha$-degroupoidification using zeroth homology instead of zeroth
cohomology.  Recall that while the zeroth cohomology of $X$ consists
of all real-valued functions on $\u{X}$, the zeroth homology consists
of such functions \textit{with finite support}.  So, we need to work
with groupoids over $X$ that give functions of this type:

\begin{definition} A groupoid $\Psi$ over $X$ is {\bf finitely supported} 
if it is tame and $\utilde{\Psi}$ is a finitely supported function on 
$\u{X}$.
\end{definition}

\noindent
Similarly, we must use spans of groupoids that give linear operators
preserving this finite support property:

\begin{definition}\label{finitetypespan}
A span:
\[\xymatrix{
& S \ar[dl]_{q} \ar[dr]^{p}& \\
Y & & X
}\]
is of {\bf finite type} if it is a tame span of groupoids and for any
finitely supported groupoid $\Psi$ over $X$, the groupoid $S\Psi$ over
$Y$ (formed by weak pullback) is also finitely supported.
\end{definition}

With these definitions, we can refine the previous propositions so
they apply to zeroth homology:

\begin{proposition}
\label{homological_alpha_degroupoidification}
Given any fixed real number $\alpha$ and a span of finite type:
\[\xymatrix{
 & S\ar[dl]_q\ar[dr]^p & \\
 Y & & X \\
}\]
there is a linear operator called its {\bf \boldmath
$\alpha$-degroupoidification}:
\[ \utilde{S_\alpha}\maps \R[\u{X}] \to \R[\u{Y}]  \]
given by:
\[   
\utilde{S_\alpha} [x] =
\sum_{[y] \in \u{Y}} \;\,
\sum_{[s]\in\u{p^{-1}(x)}\bigcap \u{q^{-1}(y)}}
\frac{|\Aut(x)|^{1-\alpha}|\Aut(y)|^\alpha}{|\Aut(s)|} \,\, [y]   \, .
\]
\end{proposition}

\begin{proposition}\label{alphaspancomposefinite}
For any $\alpha \in \R$ there is a functor from the category of
groupoids and equivalence classes of spans of finite type to the
category of real vector spaces and linear operators, sending:
\begin{itemize}
\item
each groupoid $X$ to its zeroth homology $\R[\u{X}]$, and
\item 
each span $S$ of finite type from $X$ to $Y$ to the operator
$\utilde{S_\alpha}\maps \R[\u{X}] \to \R[\u{Y}]$.
\end{itemize}
Moreover, for all values of $\alpha \in \R$, these functors are
naturally isomorphic.
\end{proposition}

\begin{proposition} 
Given a span of finite type:
\[\xymatrix{
& S \ar[dl]_{q} \ar[dr]^{p}& \\
Y & & X
}\]
and a finitely supported groupoid over $X$, say $v\maps\Psi\to X$,
then $S\Psi$ is a finitely supported groupoid
over $Y$, and $\utilde{(S\Psi)_\alpha} = 
\utilde{S_\alpha}\utilde{\Psi_\alpha}$.
\end{proposition}

The moral of this section is that we have several choices to make
before we apply degroupoidification to any specific example.  The
choice of $\alpha$ is merely a matter of convenience, but there is a
real difference between homology and cohomology, at least for
groupoids with infinitely many nonisomorphic objects.  The process
described in Section \ref{degroupoidification} is the combination of
choosing to work with cohomology and the convention $\alpha=0$ for
degroupoidifying spans. This will suffice for the majority of this
paper.  However, we will use a different choice in our study of Hall
algebras.

\section{Groupoidification}
\label{applications}

Degroupoidification is a systematic process; groupoidification is the
attempt to undo this process.  The previous section explains
degroupoidification---but not why groupoidification is interesting.
The interest lies in its applications to concrete examples.  So, let
us sketch three: Feynman diagrams, Hecke algebras, and Hall algebras.

\subsection{Feynman Diagrams}
\label{feynman}

One of the first steps in developing quantum theory was Planck's new
treatment of electromagnetic radiation.  Classically, electromagnetic
radiation in a box can be described as a collection of harmonic
oscillators, one for each vibrational mode of the field in the box.
Planck `quantized' the electromagnetic field by assuming that the energy
of each oscillator could only take discrete, evenly spaced values: if
by fiat we say the lowest possible energy is $0$, the allowed energies
take the form $n \hbar \omega$, where $n$ is any natural number,
$\omega$ is the frequency of the oscillator in question, and $\hbar$
is Planck's constant.

Planck did not know what to make of the number $n$, but Einstein and
others later interpreted it as the number of `quanta' occupying the
vibrational mode in question.  However, far from being particles in
the traditional sense of tiny billiard balls, quanta are curiously
abstract entities---for example, all the quanta occupying a given
mode are indistinguishable from each other.

In a modern treatment, states of a quantized harmonic oscillator are
described as vectors in a Hilbert space called `Fock space'.  This
Hilbert space consists of formal power series.  For a full treatment 
of the electromagnetic field we would need power series in many 
variables, one for each vibrational mode.  But to keep things simple, 
let us consider power series in one variable.  In this case, the
vector $z^n/n!$ describes a state in which $n$ quanta are present.  
A general vector in Fock space is a convergent linear combination of
these special vectors.   More precisely, the {\bf Fock space} 
consists of $\psi \in \C[[z]]$ with $\langle \psi, \psi \rangle < 
\infty$, where the inner product is given by
\begin{equation}
\label{fock_inner_product} 
         \left\langle 
         \sum a_n z^n \, , \,
         \sum b_n z^n 
         \right\rangle 
\; = \;
\sum n! \,\,  \overline{a}_n b_n \, .
\end{equation}

But what is the meaning of this inner product?  It is precisely the
inner product in $L^2(E)$, where $E$ is the groupoid of finite sets!
This is no coincidence.  In fact, there is a deep relationship between
the mathematics of the quantum harmonic oscillator and the
combinatorics of finite sets.  This relation suggests a program of
{\it groupoidifying} mathematical tools from quantum theory, such as
annihilation and creation operators, field operators and their
normal-ordered products, Feynman diagrams, and so on.  This program
was initiated by Dolan and one of the current authors
\cite{BaezDolan:2001}.  Later, it was developed much further by Morton
\cite{Morton:2006}.  Guta and Maassen \cite{GM:2002} and Aguiar and
Maharam \cite{AM:2008} have also done relevant work.  Here we just
sketch some of the basic ideas.

First, let us see why the inner product on Fock space matches the
inner product on $L^2(E)$ as described in Theorem \ref{innerprod_theorem}.
We can compute the latter inner product using a convenient basis.  Let
$\Psi_n$ be the groupoid with $n$-element sets as objects and
bijections as morphisms.  Since all $n$-element sets are isomorphic
and each one has the permutation group $S_n$ as automorphisms, we have
an equivalence of groupoids
\[               \Psi_n \simeq 1 /\!/S_n . \]
Furthermore, $\Psi_n$ is a groupoid over $E$ in an obvious way:
\[                   v \maps \Psi_n \to  E . \]
We thus obtain a vector $\utilde{\Psi}_n \in \R^{\u{E}}$
following the rule described in Definition
\ref{degroupoidification_of_vectors}.  We can describe this vector as
a formal power series using the isomorphism
\[          \R^{\u{E}} \cong \R[[z]]  
\]
described in Example \ref{power_series}.  To do this, note that
\[            v^{-1} (m) \simeq 
\begin{cases}
 1 /\!/S_n  &  m = n  \\
  0         &  m \ne n 
\end{cases}
\]
where $0$ stands for the empty groupoid.  It follows that
\[            |v^{-1} (m)| =
\begin{cases}
  1/n!  &  m = n  \\
  0         &  m \ne n 
\end{cases}
\]
and thus
\[     \utilde{\Psi}_n = \sum_{m \in \N} |v^{-1}(m)| \, z^m = 
\frac{z^n}{n!}  .\]

Next let us compute the inner product in $L^2(E)$.  Since finite
linear combinations of vectors of the form $\utilde{\Psi}_n$ are dense
in $L^2(E)$ it suffices to compute the inner product of two vectors of
this form.  We can use the recipe in Theorem \ref{innerprod_theorem}.  
So, we start by taking the weak pullback of the corresponding groupoids 
over $E$:
\[
\xymatrix{
& \ip{\Psi_m}{\Psi_n} \ar[dl] \ar[dr] & \\
\Psi_m \ar[dr] & & \Psi_n \ar[dl] \\
& E &
}
\]
An object of this weak pullback consists of an $m$-element set $S$, an
$n$-element set $T$, and a bijection $\alpha \maps S \to T$.  A
morphism in this weak pullback consists of a commutative square of
bijections:
\[
\xymatrix{
S \ar[d]_{f} \ar[r]^{\alpha} & T \ar[d]^{g} \\
S' \ar[r]_{\alpha'} & T'
}
\]
So, there are no objects in $\ip{\Psi_m}{\Psi_n}$ when $n \ne m$.
When $n = m$, all objects in this groupoid are isomorphic, and each
one has $n!$ automorphisms.  It follows that
\[       \langle \utilde{\Psi}_m, \utilde{\Psi}_n \rangle = 
| \langle \Psi_m, \Psi_n \rangle | = 
\begin{cases}
 1/n! &  m = n  \\
  0           & m \ne n 
\end{cases}
\]
Using the fact that $\utilde{\Psi}_n = z^n/n!$, we see that
this is precisely the inner product in Equation \ref{fock_inner_product}.
So, as a complex Hilbert space, Fock space is the complexification of 
$L^2(E)$.

It is worth reflecting on the meaning of the computation we just did.
The vector $\utilde{\Psi}_n = z^n/n!$ describes a state of the quantum
harmonic oscillator in which $n$ quanta are present.  Now we see that
this vector arises from the groupoid $\Psi_n$ over $E$.  In Section
\ref{intro} we called a groupoid over $E$ a {\bf stuff type}, since it
describes a way of equipping finite sets with extra stuff.  The stuff
type $\Psi_n$ is a very simple special case, where the stuff is simply
`being an $n$-element set'.  So, groupoidification reveals the
mysterious `quanta' to be simply elements of finite sets.  Moreover,
the formula for the inner product on Fock space arises from the fact
that there are $n!$ ways to identify two $n$-element sets.

The most important operators on Fock space are the annihilation and
creation operators.  If we think of vectors in Fock space as formal
power series, the {\bf annihilation operator} is given by
\[        (a \psi)(z) = \frac{d}{dz} \psi(z)   \]
while the {\bf creation operator} is given by
\[        (a^* \psi)(z) = z \psi(z)  .\]
As operators on Fock space, these are only densely defined: for example,
they map the dense subspace $\C[z]$ to itself.  However, we can also think 
of them as operators from $\C[[z]]$ to itself.  In physics these 
operators decrease or increase the number of quanta in a state, since
\[        az^n = n z^{n-1}  , \qquad a^* z^n = z^{n+1}.\]
Creating a quantum and then annihilating one is not the same as 
annhilating and then creating one, since
\[        a a^* = a^* a + 1  .\]
This is one of the basic examples of noncommutativity in quantum theory.

The annihilation and creation operators arise from spans by 
degroupoidification, using the recipe described in Theorem 
\ref{PROCESS1}.  The annihilation operator comes from this span:
\[
\xymatrix{
& E\ar[dl]_{1} \ar[dr]^{S \mapsto S+1} & \\
E & & E
}
\]
where the left leg is the identity functor and the right leg is the
functor `disjoint union with a 1-element set'.  Since it is ambiguous
to refer to this span by the name of the groupoid on top, as we have
been doing, we instead call it $A$.  Similarly, we call its adjoint
$A^*$:
\[
\xymatrix{
& E\ar[dl]_{S \mapsto S+1} \ar[dr]^{1} & \\
E & & E
}
\]
A calculation \cite{Morton:2006} shows that indeed:
\[         \utilde{A} = a, \qquad \utilde{A}^* = a^*  .\]
Moreover, we have an equivalence of spans:
\[           AA^* \simeq A^*A + 1 . \]
Here we are using composition of spans, addition of spans and the
identity span as defined in Section \ref{degroupoidification}.  If we
unravel the meaning of this equivalence, it turns out to be very
simple \cite{BaezDolan:2001}.  If you have an urn with $n$ balls in
it, there is one more way to put in a ball and then take one out than
to take one out and then put one in.  Why?  Because in the first
scenario there are $n+1$ balls to choose from when you take one out,
while in the second scenario there are only $n$.  So, the
noncommutativity of annihilation and creation operators is not a
mysterious thing: it has a simple, purely combinatorial explanation.

We can go further and define a span
\[   \Phi = A + A^*   \]
which degroupoidifies to give the well-known {\bf field operator}
\[   \phi = \utilde{\Phi} = a + a^*  \]
Our normalization here differs from the usual one in physics because
we wish to avoid dividing by $\sqrt{2}$, but all the usual physics
formulas can be adapted to this new normalization.  

The powers of the span $\Phi$ have a nice combinatorial
interpretation.  If we write its $n$th power as follows:
\[
\xymatrix{
& \Phi^n \ar[dl]_{q} \ar[dr]^{p} & \\
E & & E
}
\]
then we can reinterpret this span as a groupoid over $E \times E$:
\[
\xymatrix{
\Phi^n \ar[d]_{q \times p} \\
E \times E
}
\]
Just as a groupoid over $E$ describes a way of equipping a finite
set with extra stuff, a groupoid over $E \times E$ describes a way
of equipping a {\it pair} of finite sets with extra stuff.  
And in this example, the extra stuff in question is a very simple 
sort of diagram!  

More precisely, we can draw an object of $\Phi^n$ as a $i$-element set
$S$, a $j$-element set $T$, a graph with $i+j$ univalent vertices and
a single $n$-valent vertex, together with a bijection between the
$i+j$ univalent vertices and the elements of $S + T$.  It is against
the rules for vertices labelled by elements of $S$ to be connected by
an edge, and similarly for vertices labelled by elements of $T$.  The
functor $p \times q \maps \Phi^n \to E \times E$ sends such an object
of $\Phi^n$ to the pair of sets $(S,T) \in E \times E$.

An object of $\Phi^n$ sounds like a complicated thing, but it can be
depicted quite simply as a {\bf Feynman diagram}.  Physicists
traditionally read Feynman diagrams from bottom to top.  So, we draw
the above graph so that the univalent vertices labelled by elements of
$S$ are at the bottom of the picture, and those labelled by elements
of $T$ are at the top.  For example, here is an object of $\Phi^3$,
where $S = \{1,2,3\}$ and $T = \{4,5,6,7\}$:

\begin{center}
\begin{picture}(110,120)
  \includegraphics[scale=0.7]{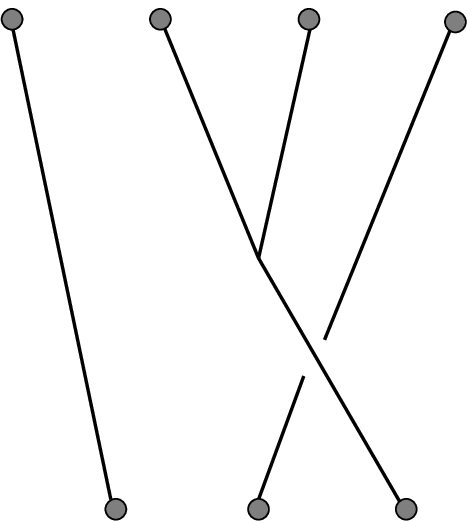} 
  \put(-96,110){$5$}
  \put(-66,110){$4$}
  \put(-36,110){$7$}
  \put(-6,110){$6$}
  \put(-76,-13){$1$}
  \put(-46,-13){$3$}
  \put(-16,-13){$2$}
\end{picture}
\end{center}

\vskip 0.5em
\noindent
In physics, we think of this as a process where 3 particles come
in and 4 go out.  

Feynman diagrams of this sort are allowed to have {\bf self-loops}:
edges with both ends at the same vertex.  So, for example, this is
a perfectly fine object of $\Phi^5$ with $S = \{1,2,3\}$ and 
$T = \{4,5,6,7\}$:

\begin{center}
\begin{picture}(110,120)
  \includegraphics[scale=0.7]{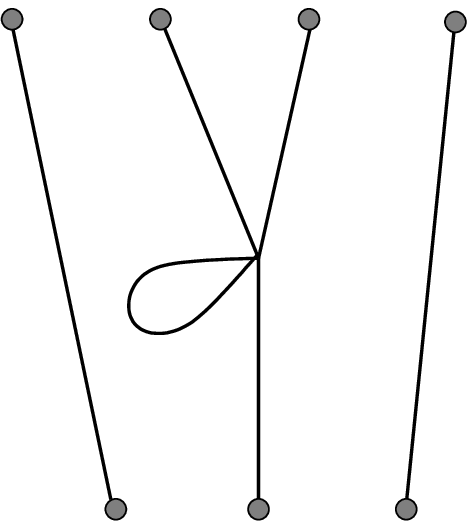}
  \put(-96,110){$5$}
  \put(-66,110){$4$}
  \put(-36,110){$6$}
  \put(-6,110){$7$}
  \put(-76,-13){$2$}
  \put(-46,-13){$3$}
  \put(-16,-13){$1$}
\end{picture}
\end{center}

\vskip 0.5em
\noindent
To eliminate self-loops, we can work with the {\bf normal-ordered
powers} or `Wick powers' of $\Phi$, denoted $\w\Phi^n\,\w\,$.  These 
are the spans obtained by taking $\Phi^n$, expanding it in terms
of the annihilation and creation operators, and moving all the 
annihilation operators to the right of all the creation operators 
`by hand', ignoring the fact that they do not commute.  For example: 
\begin{eqnarray*}     
          \w\Phi^0\,\w &=& 1    \\
          \w\Phi^1\,\w &=& A + A^*  \\
          \w\Phi^2\,\w &=& A^2 + 2A^* A + {A^*}^2 \\
          \w\Phi^3\,\w &=& A^3 + 3A^* A^2 + 3{A^*}^2 A + {A^*}^3 
\end{eqnarray*}
and so on.  Objects of $\w \Phi^n \w$ can be drawn as Feynman
diagrams just as we did for objects of $\Phi^n$.  There is just
one extra rule: self-loops are not allowed.

In quantum field theory one does many calculations involving 
products of normal-ordered powers of field operators.   Feynman 
diagrams make these calculations easy.  In the groupoidified context,
a product of normal-ordered powers is a span
\[
\xymatrix{
&       \w\Phi^{n_1}\,\w \; \cdots\; \w\Phi^{n_k}\,\w  
\ar[dl]_>>>>>>>{q} \ar[dr]^>>>>>>>{p} & \\
E & & E \, .
}
\]
As before, we can draw an object of the groupoid $\w\Phi^{n_1}\,\w \;
\cdots\; \w\Phi^{n_k}\,\w$ as a Feynman diagram.  But now these
diagrams are more complicated, and closer to those seen in physics
textbooks.  For example, here is a typical object of $\w \Phi^3 \w \,
\w \Phi^3 \w \, \w \Phi^4 \w$, drawn as a Feynman diagram:

\begin{center}
\begin{picture}(110,130)
  \includegraphics[scale=0.7]{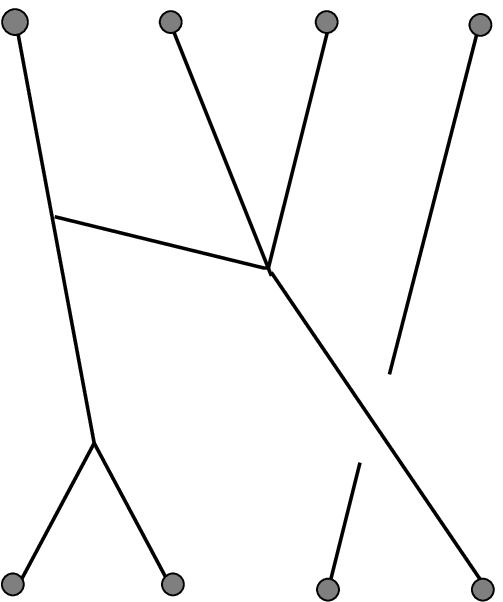}
 \put(-100,125){$5$}
  \put(-68,125){$8$}
  \put(-36,125){$7$}
  \put(-6,125){$6$}
  \put(-101,-11){$1$}
  \put(-68,-11){$4$}
  \put(-36,-11){$2$}
  \put(-5,-11){$3$} 
\end{picture}
\end{center}

\vskip 0.5em
\noindent
In general, a {\bf Feynman diagram} for an object of $\w\Phi^{n_1}\,\w
\; \cdots\; \w\Phi^{n_k}\,\w$ consists of an $i$-element set $S$, a
$j$-element set $T$, a graph with $k$ vertices of valence $n_1, \dots,
n_k$ together with $i+j$ univalent vertices, and a bijection between
these univalent vertices and the elements of $S+T$.  Self-loops are
forbidden; it is against the rules for two vertices labelled by
elements of $S$ to be connected by an edge, and similarly for two
vertices labelled by elements of $T$.  As before, the forgetful
functor $p \times q$ sends any such object to the pair of sets $(S,T)
\in E \times E$.

The groupoid $\w\Phi^{n_1}\,\w \; \cdots\; \w\Phi^{n_k}\,\w$ also contains
interesting automorphisms.  These come from {\it symmetries} of Feynman
diagrams: that is, graph automorphisms fixing the univalent
vertices labelled by elements of $S$ and $T$.  These symmetries play
an important role in computing the operator corresponding to this span:
\[
\xymatrix{
&       \w\Phi^{n_1}\,\w \; \cdots\; \w\Phi^{n_k}\,\w  
\ar[dl]_>>>>>>>{q} \ar[dr]^>>>>>>>{p} & \\
E & & E \, .
}
\]
As is evident from Theorem \ref{matrix1}, when a Feynman diagram has
symmetries, we need to divide by the number of symmetries when
determining its contribution to the operator coming from the above
span.  This rule is well-known in quantum field theory; here we see it
arising as a natural consequence of groupoid cardinality.

\subsection{Hecke Algebras}
\label{hecke}

 Here we sketch how to groupoidify a Hecke
algebra when the parameter $q$ is a power of a prime number and the
finite reflection group comes from a Dynkin diagram in this way.  More
details will appear in future work \cite{Baez}.

Let $D$ be a Dynkin diagram.  We write $d \in D$ to mean that $d$ is a
dot in this diagram.  Associated to each unordered pair of dots $d,d'
\in D$ is a number $m_{dd'} \in \{2,3,4,6\}$.  In the usual Dynkin
diagram conventions:
\begin{itemize}
\item
$m_{dd'} = 2$ is drawn as no edge at all, 
\item
$m_{dd'} = 3$ is drawn as a single edge, 
\item 
$m_{dd'} = 4$ is drawn as a double edge, 
\item
$m_{dd'} = 6$ is drawn as a triple edge.  
\end{itemize}
For any nonzero number $q$, our Dynkin diagram gives a Hecke algebra.
Since we are using real vector spaces in this paper, we work with the
Hecke algebra over $\R$:
\begin{definition}
Let $D$ be a Dynkin diagram and $q$ a nonzero real number.  The
{\bf Hecke algebra} $H(D,q)$ corresponding to this data is the associative
algebra over $\R$ with one generator $\sigma_d$ for each $d \in D$,
and relations:
\[\sigma_d^2 = (q-1)\sigma_d + q \]
for all $d \in D$, and
\[\sigma_d\sigma_{d'}\sigma_d\cdots 
= \sigma_{d'} \sigma_d \sigma_{d'}\cdots \]
for all $d,d'\in D$, where each side has $m_{dd'}$ factors.
\end{definition}
\noindent

Hecke algebras are $q$-deformations of finite reflection groups, also
known as Coxeter groups \cite{Humphreys}.  Any Dynkin diagram gives
rise to a simple Lie group, and the Weyl group of this simple Lie
algebra is a Coxeter group. 
To begin understanding Hecke algebras, it is useful to note that
when $q = 1$, the Hecke algebra is simply the group algebra of
the {\bf Coxeter group} associated to $D$: that is, the group
with one generator $s_d$ for each dot $d \in D$, and relations
\[ s_d^2 = 1,  \qquad (s_d s_{d'})^{m_{dd'}} = 1 .\]
So, the Hecke algebra can be thought of as a $q$-deformation of this
Coxeter group.

If $q$ is a power of a prime number, the Dynkin diagram $D$ determines
a simple algebraic group $G$ over the field with $q$ elements, $\F_q$.
We choose a Borel subgroup $B \subseteq G$, i.e., a maximal solvable
subgroup.  This in turn determines a transitive $G$-set $X = G/B$.
This set is a smooth algebraic variety called the {\bf flag variety}
of $G$, but we only need the fact that it is a finite set equipped
with a transitive action of the finite group $G$.  Starting from just
this $G$-set $X$, we can groupoidify the Hecke algebra $H(D,q)$.

Recalling the concept of `action groupoid' from Section \ref{intro},
we define the {\bf groupoidified Hecke algebra} to be
\[ (X \times X)\Over G. \]
This groupoid has one isomorphism class of objects for each $G$-orbit
in $X \times X$:
\[ \u{(X \times X)\Over G} \cong (X \times X)/G. \]
The well-known `Bruhat decomposition' \cite{Borel} of 
$X/G$ shows there is one such orbit for each element of the Coxeter group 
associated to $D$.  Since the Hecke algebra has a standard basis given
by elements of the Coxeter group \cite{Humphreys}, it follows that
$(X \times X)\Over G$ degroupoidifies to give the underlying vector 
space of the Hecke algebra.  In other words, we obtain an
isomorphism of vector spaces
\[   \R^{(X \times X)/G} \cong H(D,q) .\]

Even better, we can groupoidify the {\it multiplication} in the Hecke
algebra.  In other words, we can find a span that degroupoidifies to
give the linear operator
\[  
\begin{array}{ccc} 
     H(D,q) \tensor H(D,q) & \to & H(D,q)   \\
         a \tensor b       & \mapsto & ab 
\end{array}
\]
This span is very simple:
\begin{equation}
\label{multiplication}
  \def\objectstyle{\scriptstyle}
  \def\labelstyle{\scriptstyle}
   \xy
   (0,30)*+{(X \times X \times X)\Over G}="1";
   (25,10)*+{(X\times X)\Over G \;\,\times\;\, (X\times X)\Over G}="2";
   (-25,10)*+{(X\times X)\Over G}="3";
        {\ar_{(p_1,p_2)\times (p_2,p_3)} "1";"3"};
        {\ar^{(p_1,p_3)} "1";"2"};
\endxy
\end{equation}
where $p_i$ is projection onto the $i$th factor.

For a proof that this span degroupoidifies to give the desired
linear operator, see \cite{Hoffnung}.
The key is that for each dot $d \in D$ there is a special
isomorphism class in $(X \times X)\Over G$, and the function 
\[   \psi_d \maps (X \times X)/G \to \R \]
that equals 1 on this isomorphism class and 0 on the rest 
corresponds to the generator $\sigma_d \in H(D,q)$.  

To illustrate these ideas, let us consider the simplest nontrivial
example, the Dynkin diagram $A_2$:
\[\xymatrix{ \bullet\ar@{-}[r] & \bullet}\]
The Hecke algebra associated to $A_2$ has two generators, which we
call $P$ and $L$, for reasons soon to be revealed:
\[        P = \sigma_1 , \qquad L = \sigma_2 . \]
The relations are
\[     P^2 = (q-1) P + q, \qquad L^2= (q-1)P + q,  \qquad PLP=LPL  .\]
It follows that this Hecke algebra is a quotient of the group algebra
of the 3-strand braid group, which has two generators $P$ and $L$ and
one relation $PLP = LPL$, called the {\bf Yang--Baxter equation} or
{\bf third Reidemeister move}.  This is why Jones could use traces on
the $A_n$ Hecke algebras to construct invariants of knots
\cite{Jones}.  This connection to knot theory makes it especially
interesting to groupoidify Hecke algebras.

So, let us see what the groupoidified Hecke algebra looks like, and
where the Yang--Baxter equation comes from.  The algebraic group
corresponding to the $A_2$ Dynkin diagram and the prime power $q$ is
$G = \SL(3,\F_q)$, and we can choose the Borel subgroup $B$ to consist
of upper triangular matrices in $\SL(3,\F_q)$.  Recall that a {\bf
complete flag} in the vector space $\F_q^3$ is a pair of subspaces
\[           0 \subset V_1 \subset V_2 \subset \F_q^3  . \]
The subspace $V_1$ must have dimension 1, while $V_2$ must have
dimension 2.  Since $G$ acts transitively on the set of complete
flags, while $B$ is the subgroup stabilizing a chosen flag, the flag
variety $X = G/B$ in this example is just the set of complete flags in
$\F_q^3$---hence its name.

We can think of $V_1 \subset \F_q^3$ as a point in the projective
plane $\F_q{\mathrm P}^2$, and $V_2 \subset \F_q^3$ as a line in this
projective plane.  From this viewpoint, a complete flag is a chosen
point lying on a chosen line in $\F_q {\mathrm P}^2$.  This viewpoint
is natural in the theory of `buildings', where each Dynkin diagram
corresponds to a type of geometry \cite{Brown,Garrett}.  Each dot in
the Dynkin diagram then stands for a `type of geometrical figure',
while each edge stands for an `incidence relation'.  The $A_2$ Dynkin
diagram corresponds to projective plane geometry.  The dots in this
diagram stand for the figures `point' and `line':
\[\xymatrix{ {\rm point} \; \bullet \ar@{-}[r] & 
\bullet \; {\rm line} }\] 
The edge in this diagram stands for the incidence
relation `the point $p$ lies on the line $\ell$'.  

We can think of $P$ and $L$ as special elements of the $A_2$ Hecke
algebra, as already described.  But when we groupoidify the Hecke
algebra, $P$ and $L$ correspond to {\it objects} of $(X \times X)\Over
G$.  Let us describe these objects and explain how the Hecke algebra
relations arise in this groupoidified setting.

As we have seen, an isomorphism class of objects in $(X \times X)\Over
G$ is just a $G$-orbit in $X \times X$.  These orbits in turn
correspond to spans of $G$-sets from $X$ to $X$ that are {\bf
irreducible}: that is, not a coproduct of other spans of $G$-sets.
So, the objects $P$ and $L$ can be defined by giving irreducible spans
of $G$-sets:
\[
\xymatrix{
    & P \ar[dl]\ar[dr] &    &&     & L\ar[dl]\ar[dr] & \\
X &                  &X && X &                 &X
}
\]

In general, any span of $G$-sets
\[
\xymatrix{
& S \ar[dl]_{q} \ar[dr]^{p} & \\
X & & X}
\]
such that $q \times p \maps S \to X \times X$ is injective can be
thought of as $G$-invariant binary relation between elements of $X$.
Irreducible $G$-invariant spans are always injective in this sense.
So, such spans can also be thought of as $G$-invariant relations between
flags.  In these terms, we define $P$ to be the relation that says 
two flags have the same line, but different points:
\[ P = \{((p,\ell),(p',\ell)) \in X \times X \mid p\neq p'\}
\]
Similarly, we think of $L$ as a relation saying two flags
have different lines, but the same point:
\[ L = 
\{((p,\ell),(p,\ell')) \in X \times X \mid \ell\neq \ell'\}. \]
Given this, we can check that
\[ P^2 \cong (q-1) \times P + q \times 1, \qquad L^2 \cong (q-1) \times L + 
q \times 1, \qquad PLP \cong LPL . \] 
Here both sides refer to spans of $G$-sets, and we denote a span by
its apex.  Addition of spans is defined using coproduct, while $1$
denotes the identity span from $X$ to $X$.  We use `$q$' to stand for
a fixed $q$-element set, and similarly for `$q-1$'.  We compose spans
of $G$-sets using the ordinary pullback.  It takes a bit of thought to
check that this way of composing spans of $G$-sets matches the product
described by Equation \ref{multiplication}, but it is indeed the case
\cite{Hoffnung}.

To check the existence of the first two isomorphisms above, we just
need to count.  In $\mathbb{F}_q\mathrm{P}^2$, the are $q+1$ points on any
line.  So, given a flag we can change the point in $q$ different ways.
To change it again, we have a choice: we can either send it back to
the original point, or change it to one of the $q-1$ other points.
So, $P^2 \cong (q-1) \times P + q \times 1$.  Since there are also
$q+1$ lines through any point, similar reasoning shows that $L^2 \cong
(q-1) \times L + q \times 1$.

The Yang--Baxter isomorphism
\[          PLP \cong LPL  \]
is more interesting.  We construct it as follows.  First consider the
left-hand side, $PLP$.  So, start with a complete flag called
$(p_1,\ell_1)$:
\hfill \break
\begin{center}
\begin{picture}(50,25)
  \includegraphics[scale=0.35]{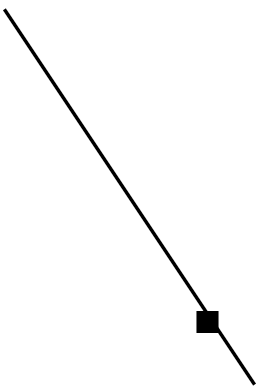}
  \put(-20,2){$p_1$}
  \put(0,-10){$\ell_1$}
\end{picture}
\end{center}
\medskip
Then, change the point to obtain a flag $(p_2,\ell_1)$.
Next, change the line to obtain a flag $(p_2,\ell_2)$.
Finally, change the point once more, which gives us the flag $(p_3,\ell_2)$:   
\begin{center}
  \begin{picture}(250,45)
  \includegraphics[scale=0.35]{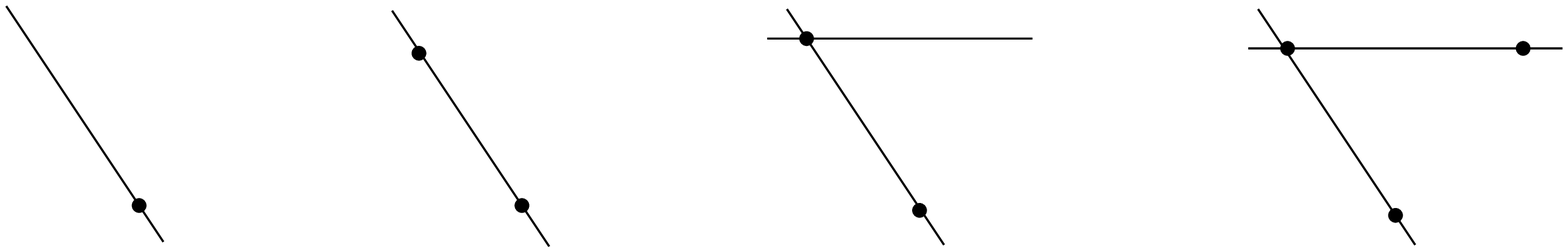}
  \put(-245,4){$p_1$}
  \put(-225,-10){$\ell_1$}
  \put(-185,4){$p_1$}
  \put(-165,-10){$\ell_1$}
  \put(-200,25){$p_2$}
  \put(-120,4){$p_1$}
  \put(-100,-10){$\ell_1$}
  \put(-130,23){$p_2$}
  \put(-80,35){$\ell_2$}
  \put(-40,4){$p_1$}
  \put(-20,-10){$\ell_1$}
  \put(-55,23){$p_2$}
  \put(5,35){$\ell_2$}
  \put(-15,40){$p_3$}
  \end{picture}
\end{center}  
\medskip
\noindent 
The figure on the far right is a typical element of $PLP$.

On the other hand, consider $LPL$.  So, start with the same flag as
before, but now change the line, obtaining $(p_1,\ell'_2)$.  Next
change the point, obtaining the flag $(p'_2,\ell'_2)$.  Finally,
change the line once more, obtaining the flag $(p'_2,\ell'_3)$:
\hfill \break
\medskip
\begin{center}
\begin{picture}(240,40)
  \includegraphics[scale=0.35]{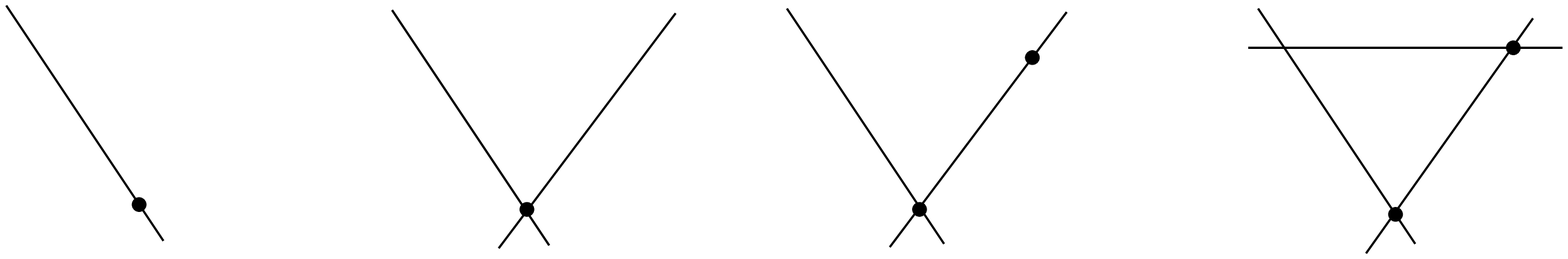}
  \put(-245,3){$p_1$}
  \put(-225,-10){$\ell_1$}
  \put(-185,3){$p_1$}
  \put(-170,-10){$\ell_1$}
  \put(-150,45){$\ell_2'$} 
  \put(-120,3){$p_1$}
  \put(-100,-10){$\ell_1$}
  \put(-85,45){$\ell_2'$} 
  \put(-85,20){$p_2'$}    
  \put(-45,4){$p_1$}
  \put(-25,-10){$\ell_1$}
  \put(-10,43){$\ell_2'$} 
  \put(-10,20){$p_2'$}  
  \put(-60,25){$\ell_3'$}   
\end{picture}
\end{center}
\medskip
\noindent
The figure on the far right is a typical element of $LPL$.

Now, the axioms of projective plane geometry say that any two distinct
points lie on a unique line, and any two distinct lines intersect in a
unique point.  So, any figure of the sort shown on the left below
determines a unique figure of the sort shown on the right, and vice
versa: 
\hfill\break
\begin{center}
\begin{picture}(150,50)
\includegraphics[scale=0.50]{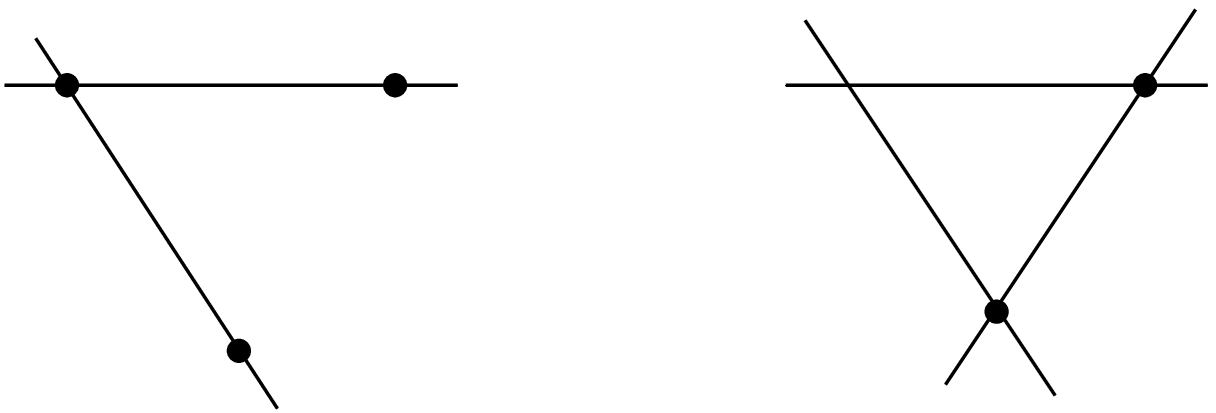}
\end{picture}
\end{center}
\medskip
\noindent
Comparing this with the pictures above, we see this bijection induces
an isomorphism of spans $PLP \cong LPL$.  So, we have derived the
Yang--Baxter isomorphism from the axioms of projective plane geometry!

To understand groupoidified Hecke algebras, it is important to keep
straight the two roles played by spans.  On the one hand, objects of
the groupoidified Hecke algebra $(X \times X)/\!/G$ can be described
as certain spans from $X$ to $X$, namely the injective $G$-invariant
ones.  Multiplying these objects then corresponds to composing spans.
On the other hand, Equation \ref{multiplication} gives a span describing the
multiplication in $(X \times X)/\!/G$.  In fact, this span
describes the process of composing spans.  If this seems hopelessly
confusing, remember that any matrix describes a linear operator, but
there is also a linear operator describing the process of matrix
multiplication.  We are only groupoidifying that idea.

Other approaches to categorified Hecke algebras and their
representations have been studied by a number of authors, building on
Kazhdan--Lusztig theory \cite{KaLu}.  One key step was Soergel's
introduction of what are nowadays called Soergel bimodules
\cite{Rou,Soergel}.  Also important were Khovanov's categorification
of the Jones polynomial \cite{Kh} and the work by Bernstein, Frenkel,
Khovanov and Stroppel on categorifying Temperley--Lieb algebras,
which are quotients of the type $A$ Hecke algebras
\cite{BFK,Stroppel}.  A diagrammatic interpretation of the Soergel
bimodule category was developed by Elias and Khovanov \cite{ElKh}, and
a geometric approach led Webster and Williamson \cite{WW} to deep
applications in knot homology theory.  This geometric interpretation
can be seen as going beyond the simple form of groupoidification we
consider here, and considering groupoids in the category of schemes.

\subsection{Hall Algebras}\label{hall}

The Hall algebra of a quiver is a very natural example of
groupoidification, and a very important one, since it lets us 
groupoidify `half of a quantum group'.  However, to obtain the usual
formula for the Hall algebra product, we need to exploit one of the
alternative conventions explained in Section \ref{homology}.  In this
section we begin by quickly reviewing the usual theory of Hall
algebras and their relation to quantum groups \cite{Hubery,Schiffman}.
Then we explain how to groupoidify a Hall algebra.

We start by fixing a finite field $\F_q$ and a directed graph $D$, 
which might look like this:
\[ \xymatrix@=10pt{
 &&&&& \bullet  \\
 \bullet \ar@(ul,dl)[] \ar@/^1pc/[rr]
 && \bullet \ar@/^1pc/[ll] \ar@/^1pc/[rr] \ar@/_1pc/[rr] \ar[rr]
 && \bullet \ar[ur] \ar[dr] \\
 &&&&& \bullet  
 }
\]
We shall call the category $Q$ freely generated by $D$ a \textbf{quiver}.
The objects of $Q$ are the vertices of $D$, while the morphisms are
edge paths, with paths of length zero serving as identity morphisms.

By a \textbf{representation} of the quiver $Q$ we mean a functor 
\[   R \maps Q \to \FinVect_q, \]
where $\FinVect_q$ is the category of finite-dimensional vector spaces
over $\F_q$.  Such a representation simply assigns a vector space
$R(d) \in \FinVect_q$ to each vertex of $D$ and a linear operator
$R(e) \maps R(d) \to R(d')$ to each edge $e$ from $d$ to $d'$.  By a
\textbf{morphism} betwen representations of $Q$ we mean a natural
transformation between such functors.  So, a morphism $\alpha \maps R
\to S$ assigns a linear operator $\alpha_d \maps R(d) \to S(d)$ to
each vertex $d$ of $D$, in such a way that
\[
\xymatrix{
R(d) \ar[d]_{\alpha_d} \ar[r]^{R(e)} & R(d') \ar[d]^{\alpha_{d'}} \\
S(d) \ar[r]_{S(d)} & S(d')
}
\]
commutes for any edge $e$ from $d$ to $d'$.  There is a category
$\Rep(Q)$ where the objects are representations of $Q$ and the
morphisms are as above.  This is an abelian category, so we can speak
of indecomposable objects, short exact sequences, etc.\ in this
category.

In 1972, Gabriel \cite{Gabriel} discovered a remarkable fact.  Namely:
a quiver has finitely many isomorphism classes of indecomposable
representations if and only if its underlying graph, ignoring the
orientation of edges, is a finite disjoint union of Dynkin diagrams of
type $A, D$ or $E$.  These are called {\bf simply laced} Dynkin
diagrams.

Henceforth, for simplicity, we assume the underlying graph of our
quiver $Q$ is a simply laced Dynkin diagram when we ignore the
orientations of its edges.  Let $X$ be the underlying groupoid of
$\Rep(Q)$: that is, the groupoid with representations of $Q$ as
objects and \textit{isomorphisms} between these as morphisms.  We will
use this groupoid to construct the Hall algebra of $Q$.

As a vector space, the Hall algebra is just $\R[\u{X}]$.
Remember from Section \ref{homology} that this is the vector space
whose basis consists of isomorphism classes of objects in $X$.  In
fancier language, it is the zeroth homology of $X$.  So, we should use
the homology approach to degroupoidification, instead of the
cohomology approach used in our examples so far.

We now focus our attention on the Hall algebra product.  Given three quiver
representations $M,N,$ and $E$, we define:
\[\mathcal{P}_{MN}^E=
\{(f,g): 
0\to N \stackrel{f}{\rightarrow} E \stackrel{g}{\rightarrow} M \to 0 
\textrm{\; is exact} \}. \]
The Hall algebra product counts these exact sequences, but with a
subtle `correction factor':
\[[M] \cdot [N] =\sum_{E \in \u{X}} 
\frac{|\mathcal{P}_{MN}^E|}{|\Aut(M)| \, |\Aut(N)|}\, [E] \,.\]
All the cardinalities in this formula are ordinary \textit{set}
cardinalities.  

Somewhat suprisingly, the above product is associative.  In fact,
Ringel \cite{Ringel} showed that the resulting algebra is isomorphic
to the positive part $U_q^+ \g$ of the quantum group corresponding to
our simply laced Dynkin diagram!  So, roughly speaking, the Hall algebra
of a simply laced quiver is `half of a quantum group'.

Since the Hall algebra product can be seen as a linear operator
\[  
\begin{array}{ccc} 
 \R[\u{X}] \tensor \R[\u{X}] &\to& \R[\u{X}] \\
         a \tensor b       & \mapsto & a\cdot b 
\end{array}
\]
it is natural to seek a span of groupoids 
\[
\xymatrix@!C{
& {\rm ???}\ar[dl]_{q} \ar[dr]^{p} & \\
X & & X\times X
}
\]
that gives this operator.  Indeed, there is a very natural span that
gives this product.  This will allow us to groupoidify the algebra
$U_q^+ \g$.  

We start by defining a groupoid $\SES(Q)$ to serve as the apex of this
span.  An object of $\SES(Q)$ is a short exact sequence in $\Rep(Q)$,
and a morphism from
\[ 0\to N \stackrel{f}{\to} E \stackrel{g}{\to} M\to 0 \]
to
\[ 0\to N' \stackrel{f'}{\to} E' \stackrel{g'}{\to} M'\to 0 \]
is a commutative diagram
\[
\xymatrix{
0 \ar[r] & 
N \ar[r]^{f} \ar[d]^{\alpha} & 
E \ar[r]^{g} \ar[d]^{\beta} & 
M \ar[r] \ar[d]^{\gamma} & 0 \\
0 \ar[r] & N' \ar[r]^{f'} & E' \ar[r]^{g'} & M' \ar[r] & 0 \\
}
\]
where $\alpha,\beta,$ and $\gamma$ are isomorphisms of
quiver representations.

Next, we define the span
\[
\xymatrix@!C{
& \SES(X)\ar[dl]_{q} \ar[dr]^{p} & \\
X & & X\times X
}
\]
where $p$ and $q$ are given on objects by
\[
\begin{array}{ccl}
  p(0\to N \stackrel{f}{\to} E \stackrel{g}{\to} M\to 0) &=& (M,N) \\
  q(0\to N \stackrel{f}{\to} E \stackrel{g}{\to} M\to 0) &=& E 
\end{array}
\]
and defined in the natural way on morphisms.  This span captures the
idea behind the standard Hall algebra multiplication.  Given two
quiver representations $M$ and $N$, this span relates them to every 
representation $E$ that is an extension of $M$ by $N$.

Before we degroupoidify this span, we need to decide on a convention.
It turns out that the correct choice is $\alpha$-degroupoidification
with $\alpha = 1$, as described in Section \ref{homology}.  Recall
from Proposition \ref{homological_alpha_degroupoidification} that a
span of finite type
\[\xymatrix{
 & S\ar[dl]_q\ar[dr]^p & \\
 Y & & X \\
}\]
yields an operator 
\[ \utilde{S_1}\maps \R[\u{X}] \to \R[\u{Y}]  \]
given by:
\[
\utilde{S_1} [x] =
\sum_{[y] \in \u{Y}} \;\,
\sum_{[s]\in\u{p^{-1}(x)}\bigcap \u{q^{-1}(y)}}
\frac{|\Aut(y)|}{|\Aut(s)|} \,\, [y]   \, .
\]
We can rewrite this using groupoid cardinality as follows:
\[
\utilde{S_1} [x] =
\sum_{[y] \in \u{Y}} \;\,
|\Aut(y)| \,\, |(p \times q)^{-1}(x,y)| \,\, [y]   \, .
\]
Applying this procedure to the span with $\SES(Q)$ as its apex, we get
an operator
\[  m \maps \R[\u{X}] \tensor \R[\u{X}] \to \R[\u{X}] \]
with
\[ m([M] \tensor [N]) 
= \sum_{E\in\mathcal{P}^E_{MN}}
|\Aut(E)| \, |(p\times q)^{-1}(M,N,E)|  \, \, [E] .  \]
We wish to show this matches the Hall algebra product $[M] \cdot [N]$.

For this, we must make a few observations. First, we note that the
group $\Aut(N) \times \Aut(E) \times \Aut(M)$ acts on the set
$\mathcal{P}^E_{MN}$. This action is not necessarily free, but this is
just the sort of situation groupoid cardinality is designed to handle.
Taking the weak quotient, we obtain a groupoid equivalent to the
groupoid where objects are short exact sequences of the form $0\to
N\to E\to M\to 0$ and morphisms are isomorphisms of short exact
sequences.  So, the weak quotient is equivalent to the groupoid
$(p\times q)^{-1}(M,N,E)$.  Remembering that groupoid cardinality is
preserved under equivalence, we see:
\[
\begin{array}{rcl}
|(p\times q)^{-1}(M,N,E)| &=& 
|\mathcal{P}^E_{MN}/\!/(\Aut(N)\times\Aut(E)\times\Aut(M))| \\  \\
&=& 
\displaystyle{\frac{|\mathcal{P}^E_{MN}|}
{|\Aut(N)| \, |\Aut(E)| \, |\Aut(M)|} . }
\end{array}
\]
So, we obtain
\[ m([M] \tensor [N]) 
= \sum_{E\in\mathcal{P}^E_{MN}}
\frac{|\mathcal{P}^E_{MN}|} {|\Aut(M)| \, |\Aut(N)|}  \,\, [E] .  \]
which is precisely the Hall algebra product $[M] \cdot [N]$.

We can define a coproduct on $\R[\u{X}]$ using the the adjoint of the
span that gives the product.  Unfortunately this coproduct does not
make the Hall algebra into a bialgebra (and thus not a Hopf algebra).
Ringel discovered how to fix this problem by `twisting' the product
and coproduct \cite{Ringel2}.  The resulting twisted Hall algebra is
isomorphic \textit{as a Hopf algebra} to $U_q^+ \g$.  This adjustment
also removes the dependency on the direction of the arrows in our
original directed graph.  We hope to groupoidify this construction in
future work.

\section{Degroupoidifying a Tame Span}
\label{processes}

In Section \ref{degroupoidification} we described a process for
turning a tame span of groupoids into a linear operator.  In this
section we show this process is well-defined.  The calculations
in the proof yield an explicit criterion for when a span is tame.
They also give an explicit formula for the the operator coming from a
tame span.  As part of our work, we also show that equivalent spans
give the same operator.

\subsection{Tame Spans Give Operators}

To prove that a tame span gives a well-defined operator, we begin with
three lemmas that are of some interest in themselves.  We postpone to
Appendix \ref{appendix} some well-known facts about groupoids that do
not involve the concept of degroupoidification.  This Appendix also
recalls the familiar concept of `equivalence' of groupoids, which
serves as a basis for this:

\begin{definition}
Two groupoids over a fixed groupoid $X$, say $v \maps \Psi \to X$
and $w \maps \Phi \to X$, are {\bf equivalent} as groupoids 
over $X$ if there is an equivalence $F \maps \Psi \to \Phi$ 
such that this diagram
\[
\xymatrix{
\Psi \ar[rr]^{F} \ar[dr]_{p} && \Phi \ar[dl]^{q} \\
&X&
}
\]
commutes up to natural isomorphism.
\end{definition}

\begin{lemma}\label{vectorswelldefined}
Let $v \maps \Psi \to X$ and $w \maps \Phi \to X$ be equivalent
groupoids over $X$.  If either one is tame, then both are tame,
and $\utilde{\Psi} = \utilde{\Phi}$.
\end{lemma}

\begin{proof}
This follows directly from Lemmas \ref{EQUIVGRPD} and 
\ref{ESSENTIALPULLBACK} in Appendix \ref{appendix}.
\end{proof}

\begin{lemma}
\label{addvectors}
Given tame groupoids $\Phi$ and $\Psi$ over $X$,
\[\utilde{\Phi + \Psi} = \utilde{\Phi} + \utilde{\Psi}.\]
More generally, given any collection of tame groupoids $\Psi_i$ over $X$,
the coproduct $\sum_i \Psi_i$ is naturally a 
groupoid over $X$, and if it is tame, then
\[\utilde{\sum_i \Psi_i} = \sum_i \utilde{\Psi}_i \]
where the sum on the right hand side converges pointwise as
a function on $\u{X}$.
\end{lemma}

\begin{proof}
The full inverse image of any object $x \in X$ in the coproduct
$\sum_i \Psi_i$ is the coproduct of its fulll inverse images in
each groupoid $\Psi_i$.  Since groupoid cardinality is additive under
coproduct, the result follows.
\end{proof}

\begin{lemma}
\label{linearity}
Given a span of groupoids 
\[
\xymatrix{
& S\ar[dl]_{q} \ar[dr]^{p} & \\
Y & & X
}
\]
we have
\begin{enumerate}
\item $S(\sum_i \Psi_i)\simeq \sum_i S\Psi_i$
\item $S(\Lambda\times \Psi)\simeq \Lambda\times S\Psi$
\end{enumerate}
whenever $v_i \maps \Psi_i \to X$ are groupoids over $X$,
$v \maps \Psi \to X$ is a groupoid over $X$, and 
$\Lambda$ is a groupoid.
\end{lemma}

\begin{proof}
To prove 1, we need to describe a functor
\[ 
F\maps \sum_i S \Psi_i \to S(\sum_i \Psi_i)
\]
that will provide our equivalence.  For this, we simply need to
describe for each $i$ a functor $F_i \maps S \Psi_i \to S(\sum_i
\Psi_i)$.  An object in $S \Psi_i$ is a triple $(s,z,\alpha)$ where $s
\in S$, $z \in \Psi_i$ and $\alpha \maps p(s) \to v_i(z)$.  $F_i$
simply sends this triple to the same triple regarded as an object of
$S(\sum_i \Psi_i)$.  One can check that $F$ extends to a functor and
that this functor extends to an equivalence of groupoids over $S$.

To prove 2, we need to describe a functor $F \maps S(\Lambda\times
\Phi) \to \Lambda\times S\Phi$.  This functor simply re-orders the
entries in the quadruples which define the objects in each groupoid.
One can check that this functor extends to an equivalence of groupoids
over $X$.
\end{proof}

\noindent
Finally we need the following lemma, which simplifies the computation
of groupoid cardinality:

\begin{lemma}\label{ALTCARD}
If $X$ is a tame groupoid with finitely many objects in each isomorphism
class, then 
\[ |X| 
= \sum_{x \in X} \frac{1}{|\Mor(x,-)|} \]
where $\Mor(x,-) = \bigcup_{y\in X}\hom(x,y)$ is
the set of morphisms whose source is the object $x \in X$.  
\end{lemma}

\begin{proof}
We check the following equalities:
\[  \sum_{[x] \in \u{X}} \frac{1}{|\Aut(x)|} 
= \sum_{[x] \in \u{X}} \frac{|[x]|}{|\Mor(x,-)|} 
= \sum_{x \in X} \frac{1}{|\Mor(x,-)|} .\]
Here $[x]$ is the set of objects isomorphic to $x$, and $|[x]|$ is the
ordinary cardinality of this set.  To check the above equations, we first
choose an isomorphism $\gamma_y \maps x \to y$ for each object $y$ isomorphic
to $x$.  This gives a bijection from $[x] \times 
\Aut(x)$ to $\Mor(x,-)$ that takes $(y, f \maps x \to x)$ to
$\gamma_y f \maps x \to y$.  Thus
\[         |[x]| \, |\Aut(x)| = |\Mor(x,-)| , \]
and the first equality follows.  We also get a bijection between
$\Mor(y,-)$ and $\Mor(x,-)$ that takes $f \maps y \to z$ to
$f\gamma_y \maps x \to z$.  Thus, $|\Mor(y,-)| = |\Mor(x,-)|$ whenever
$y$ is isomorphic to $x$.  The second equation follows from this.
\end{proof}

Now we are ready to prove the main theorem of this section:

\begin{thm}\label{PROCESS1}
Given a tame span of groupoids
\[
\xymatrix{
& S \ar[dl]_{q} \ar[dr]^{p} & \\
Y && X
}
\]
there exists a unique linear operator 
$\utilde{S} \maps \R^{\u{X}} \to \R^{\u{Y}}$
such that $\utilde{S}\utilde{\Psi} = \utilde{S\Psi}$ for 
any vector $\utilde{\Psi}$ obtained from a tame groupoid $\Psi$
over $X$.
\end{thm}

\begin{proof}
It is easy to see that these conditions uniquely determine
$\utilde{S}$.  Suppose $\psi \maps \u{X} \to \R$ is
any nonnegative function.  Then we can find a groupoid $\Psi$
over $X$ such that $\utilde{\Psi} = \psi$.  So, $\utilde{S}$
is determined on nonnegative functions by the condition that
$\utilde{S}\utilde{\Psi} = \utilde{S\Psi}$.  Since every
function is a difference of two nonnegative functions and 
$\utilde{S}$ is linear, this uniquely determines $\utilde{S}$.

The real work is proving that $\utilde{S}$ is well-defined.  
For this, assume we have a collection $\lbrace v_i \maps \Psi_i \to X
\rbrace_{i\in I}$ of groupoids over $X$ and real numbers $\lbrace
\alpha_i \in \R \rbrace_{i\in I}$ such that
\begin{equation}
\label{assumption1}
 \sum_i \alpha_i \, \utilde{\Psi_i} = 0. 
\end{equation}
We need to show that
\begin{equation}
\label{fact0}
 \sum_i\alpha_i \, \utilde{S\Psi_i} = 0.
\end{equation}

We can simplify our task as follows.  First, recall that a {\bf
skeletal} groupoid is one where isomorphic objects are equal.  Every
groupoid is equivalent to a skeletal one.  Thanks to Lemmas
\ref{vectorswelldefined} and \ref{skeletal}, we may therefore assume
without loss of generality that $S$, $X$, $Y$ and all the groupoids
$\Psi_i$ are skeletal.

Second, recall that a skeletal groupoid is a coproduct of groupoids
with one object.  By Lemma \ref{addvectors}, degroupoidification
converts coproducts of groupoids over $X$ into sums of vectors.  Also,
by Lemma \ref{linearity}, the operation of taking weak pullback
distributes over coproduct.  As a result, we may assume without loss
of generality that each groupoid $\Psi_i$ has one object.  Write
$\ast_i$ for the one object of $\Psi_i$.

With these simplifying assumptions, Equation \ref{assumption1} says
that for any $x \in X$,
\begin{equation}
\label{assumption2}
 0  =  \displaystyle{\sum_{i \in I} \alpha_i \, \utilde{\Psi_i}([x])} 
    =  \displaystyle{\sum_{i \in I} \alpha_i \, |v_i^{-1} (x)|}   
    = \displaystyle{\sum_{i \in J} \frac{\alpha_i}{|\Aut(\ast_i)|}}
\end{equation}
where $J$ is the collection of $i \in I$ such that $v_i(\ast_i)$ is
isomorphic to $x$.  Since all groupoids in sight
are now skeletal, this condition implies $v_i(\ast_i) = x$.

Now, to prove Equation \ref{fact0}, we need to show that 
\[ \sum_{i \in I} \alpha_i \, \utilde{S\Psi_i}([y]) = 0\]
for any $y \in Y$.  But since the set $I$ is partitioned into sets
$J$, one for each $x \in X$, it suffices to show
\begin{equation}
\label{fact1}
\sum_{i \in J} \alpha_i \, \utilde{S\Psi_i}([y]) = 0 .
\end{equation}
for any fixed $x \in X$ and $y \in Y$.

To compute $\utilde{S\Psi_i}$, we need
to take this weak pullback:
\[
\xymatrix{
& & S\Psi_i \ar[dl]_{\pi_S}\ar[dr]^{\pi_{\Psi_i}} & \\
& S \ar[dl]_{q} \ar[dr]^{p} & & \Psi_i \ar[dl]_{v_i} \\
Y & & X &  
}
\]
We then have
\begin{equation}
\label{formula1}
    \utilde{S \Psi_i}([y]) = |(q \pi_S)^{-1}(y)| , 
\end{equation}
so to prove Equation \ref{fact1} it suffices to show
\begin{equation}
\label{fact2}
\sum_{i \in J} \alpha_i \, |(q \pi_S)^{-1}(y)| = 0 .
\end{equation}

Using the definition of `weak pullback', and taking advantage of the
fact that $\Psi_i$ has just one object, which maps down to $x$, we can
see that an object of $S \Psi_i$ consists of an object $s \in S$ with
$p(s) = x$ together with an isomorphism $\alpha \maps x \to x$.  This
object of $S\Psi_i$ lies in $(q \pi_S)^{-1}(y)$ precisely when we also
have $q(s) = y$.

So, we may briefly say that an object of $(q\pi_S)^{-1}(y)$ is a pair
$(s, \alpha)$, where $s \in S$ has $p(s) = x$, $q(s) = y$, and
$\alpha$ is an element of $\Aut(x)$.  Since $S$ is skeletal, there is
a morphism between two such pairs only if they have the same first
entry.  A morphism from $(s,\alpha)$ to $(s,\alpha')$ then consists of
a morphism $f \in \Aut(s)$ and a morphism $g \in \Aut(\ast_i)$ such
that
\[
\xymatrix{
x \ar[r]^{\alpha}\ar[d]_{p(f)} & x\ar[d]^{v_i(g)} \\
x \ar[r]_{\alpha'} & x
}
\]
commutes.  

A morphism out of $(s,\alpha)$ thus consists of an arbitrary pair $f
\in \Aut(s)$, $g \in \Aut(\ast_i)$, since these determine the target
$(s,\alpha')$.  This fact and Lemma \ref{ALTCARD} allow us to compute:
\begin{equation}
\label{formula2}
\begin{array}{ccl}
\nonumber |(q\pi_S)^{-1}(y)| & = 
\displaystyle{ 
\sum_{(s,\alpha)\in (q\pi_S)^{-1}(y)}
\frac{1}{|\Mor((s,\alpha),-)|} 
}
\\
\\
& = 
\displaystyle{
\sum_{s \in p^{-1}(y) \cap q^{-1}(y)}
\frac{|\Aut(x)|}{|\Aut(s)||\Aut(\ast_i)|} \,.
} 
\end{array} 
\end{equation}
So, to prove Equation \ref{fact2}, it suffices to show
\begin{equation}
\label{fact3}
\sum_{i \in J} \;\, \sum_{s \in p^{-1}(x) \cap q^{-1}(y) }
\frac{\alpha_i|\Aut(x)|}{|\Aut(s)||\Aut(\ast_i)|} = 0 \,.
\end{equation}
But this easily follows from Equation \ref{assumption2}.  So, the
operator $\utilde{S}$ is well defined.
\end{proof}

In Definition \ref{EQUIVALENCE} we recall the natural concept of
`equivalence' for spans of groupoids.  The next theorem says that our
process of turning spans of groupoids into linear operators sends
equivalent spans to the same operator:

\begin{thm}\label{linops_equivspans1}
Given equivalent spans
\[
\xymatrix{
& S\ar[dl]_{q_S} \ar[dr]^{p_S} &  & & T\ar[dl]_{q_T} \ar[dr]^{p_T} &\\
Y & & X & Y & & X
}
\]
the linear operators $\utilde{S}$ and $\utilde{T}$ are equal.
\end{thm}

\begin{proof}
Since the spans are equivalent, there is a functor providing an
equivalence of groupoids $F \maps S \to T$ along with a pair of
natural isomorphisms $\alpha \maps p_S \Rightarrow p_TF$ and 
$\beta \maps q_S \Rightarrow q_TF$.  Thus, the diagrams
\[
\xymatrix{
S \ar[dr] && \Phi \ar[dl] & & T \ar[dr] && \Phi \ar[dl] \\
& X & && & X &
}
\]
are equivalent pointwise.  It follows from Lemma \ref{skeletal} that
the weak pullbacks $S\Psi$ and $T\Psi$ are equivalent groupoids with
the equivalence given by a functor $\tilde{F} \maps S\Psi \to T\Psi$.
From the universal property of weak pullbacks, along with $F$, we
obtain a natural transformation $\gamma \maps F\pi_S \Rightarrow
\pi_T\tilde{F}$.  We then have a triangle
\[ \def\objectstyle{\scriptstyle}
  \def\labelstyle{\scriptstyle}
   \xy
   (20,20)*+{S\Psi}="1";
   (-20,20)*+{T\Psi}="2";
   (10,0)*+{S}="3";
   (-10,0)*+{T}="4";
   (0,-20)*+{Y}="5";
        {\ar_{\tilde{F}} "1";"2"};
        {\ar^{\pi_S} "1";"3"};
        {\ar_{\pi_T} "2";"4"};
        {\ar^{q_S} "3";"5"};
        {\ar_{q_T} "4";"5"};
        {\ar_{F} "3";"4"};
        {\ar@{=>}_<<{\scriptstyle \gamma} (2,11); (-2,9)};
        {\ar@{=>}_<<{\scriptstyle \beta} (2,-9); (-2,-11)};

\endxy
\]
where the composite of $\gamma$ and $\beta$ is $(q_T \cdot
\gamma)^{-1}\beta \maps q_S \pi_S \Rightarrow q_T\pi_T\tilde{F}$.
Here $\cdot$ stands for whiskering: see Definition \ref{whiskering}.

We can now apply Lemma \ref{ESSENTIALPULLBACK}.  Thus, for every $y \in
Y$, the full inverse images $(q_S\pi_S)^{-1}(y)$ and
$(q_T\pi_T)^{-1}(y)$ are equivalent.  It follows from Lemma
\ref{EQUIVGRPD} that for each $y \in Y$, the groupoid cardinalities
$|(q_S\pi_S)^{-1}(y)|$ and $|(q_T\pi_T)^{-1}(y)|$ are equal.  Thus,
the linear operators $\utilde{S}$ and $\utilde{T}$ are the same.
\end{proof}

\subsection{An Explicit Formula}
\label{explicit}

Our calculations in the proof of Theorem \ref{PROCESS1} yield an explicit
formula for the operator coming from a tame span, and a criterion for
when a span is tame:

\begin{thm}
\label{matrix1}
A span of groupoids
\[
\xymatrix{
& S\ar[dl]_{q} \ar[dr]^{p} & \\
Y & & X
}
\]
is tame if and only if:
\begin{enumerate}
\item
For any object $y\in Y$, the groupoid $p^{-1}(x)\cap q^{-1}(y)$ is
nonempty for objects $x$ in only a finite number of isomorphism classes
of $X$.
\item For every $x \in X$ and $y \in Y$, the groupoid
$p^{-1}(x)\cap q^{-1}(y)$ is tame. 
\end{enumerate}
Here $p^{-1}(x)\cap q^{-1}(y)$ is the subgroupoid of $S$ whose objects
lie in both $p^{-1}(x)$ and $q^{-1}(y)$, and whose morphisms lie in
both $p^{-1}(x)$ and $q^{-1}(y)$.

If $S$ is tame, then for any $\psi \in \R^{\u X}$ we
have
\[   
(\utilde{S} \psi)([y]) = 
\sum_{[x] \in \u{X}} \;\,
\sum_{[s]\in\u{p^{-1}(x)}\bigcap \u{q^{-1}(y)}}
\frac{|\Aut(x)|}{|\Aut(s)|} \,\, \psi([x]) \,. 
\]
\end{thm}

\begin{proof} First suppose the span $S$ is tame and $v \maps \Psi \to X$ 
is a tame groupoid over $X$.  Equations \ref{formula1} and
\ref{formula2} show that if $S,X,Y,$ and $\Psi$ are skeletal, and
$\Psi$ has just one object $\ast$, then
\[   \utilde{S \Psi}([y]) = \sum_{s \in p^{-1}(x) \cap q^{-1}(y)}
\frac{|\Aut(v(\ast))|}{|\Aut(s)| |\Aut(\ast)|}  \]
On the other hand, 
\[   \utilde{\Psi}([x]) = 
\begin{cases}\displaystyle{\frac{1}{|\Aut(\ast)|}} 
              & \textrm{if} \; v(\ast) = x \\ 
                                           \\
              0                     & \textrm{otherwise.}  
\end{cases}
\]
So in this case, writing $\utilde{\Psi}$ as $\psi$, we have
\[   
(\utilde{S} \psi)([y]) = 
\sum_{[x] \in X} \;\,
\sum_{[s]\in p^{-1}(x) \bigcap q^{-1}(y)}
\frac{|\Aut(x)|}{|\Aut(s)|} \,\, \psi([x]) \,. 
\]
Since both sides are linear in $\psi$, and every nonnegative function
in $\R^{\u{X}}$ is a pointwise convergent nonnegative linear
combination of functions of the form $\psi = \utilde{\Psi}$ with
$\Psi$ as above, the above equation in fact holds for {\it all} $\psi
\in \R^{\u{X}}$.

Since all groupoids in sight are skeletal, we may equivalently 
write the above equation as
\[   
(\utilde{S} \psi)([y]) = 
\sum_{[x] \in \u{X}} \;\,
\sum_{[s]\in\u{p^{-1}(x)}\bigcap \u{q^{-1}(y)}}
\frac{|\Aut(x)|}{|\Aut(s)|} \,\, \psi([x]) \,. 
\]
The advantage of this formulation is that now both sides are 
unchanged when we replace $X$ and $Y$ by equivalent groupoids,
and replace $S$ by an equivalent span.   So, this equation holds
for all tame spans, as was to be shown.

If the span $S$ is tame, the sum above must converge for all 
functions $\psi$ of the form $\psi = \utilde{\Psi}$.  Any 
nonnegative function $\psi \maps \u{X} \to \R$ is of this form.
For the sum above to converge for {\it all} nonnegative $\psi$, 
this sum:
\[  \sum_{[s]\in\u{p^{-1}(x)}\bigcap \u{q^{-1}(y)}}
\frac{|\Aut(x)|} {|\Aut(s)|}\]
must have the following two properties:
\begin{enumerate}
\item
For any object $y \in Y$, it is nonzero only for objects $x$ in 
a finite number of isomorphism classes of $X$.
\item For every $x \in X$ and $y \in Y$, it converges to a finite
number.
\end{enumerate}
These conditions are equivalent to conditions 1) and 2) in the 
statement of the theorem.  We leave it as an exercise to check that
these conditions are not only necessary but also sufficient for 
$S$ to be tame.
\end{proof}

The previous theorem has many nice consequences.  For example:

\begin{proposition}
\label{add_spans}
Suppose $S$ and $T$ are tame spans from a groupoid $X$ to a groupoid
$Y$.  Then $\utilde{S+T} = \utilde{S} + \utilde{T}$.
\end{proposition}

\begin{proof}
This follows from the explicit formula given in Theorem \ref{matrix1}.
\end{proof}

\section{Properties of Degroupoidification}
\label{properties}

In this section we prove all the remaining results stated in Section
\ref{degroupoidification}.  We start with results about scalar
multiplication.  Then we show that degroupoidification is a functor.
Finally, we prove the results about inner products and adjoints.

\subsection{Scalar Multiplication}
\label{scalar_mult}

To prove facts about scalar multiplication, we use the following lemma:

\begin{lemma}\label{essinverseproduct} 
Given a groupoid $\Lambda$ and a functor between groupoids $p\maps X\to Y$,
then the functor $c\times p \maps \Lambda\times Y\to 1\times X$ (where
$c\maps \Lambda\to 1$ is the unique morphism from $\Lambda$ to the terminal
groupoid $1$) satisfies:
\[ |(c\times p)^{-1}(1, x)|=|\Lambda| \, |p^{-1}(x)| \]
for all $x \in X$.
\end{lemma}

\begin{proof}
By the definition of full inverse image we have
\[(c\times p)^{-1}(1, x) \iso \Lambda \times p^{-1}(x)  .\]
In the product $\Lambda \times p^{-1}(x)$, an automorphism of an
object $(\lambda,y)$ is an automorphism of $\lambda$ together with
an automorphism of $y$.  We thus obtain
\[|(c\times p)^{-1}(1, x)| = 
\sum_{[\lambda] \in \u{\Lambda}} \;
\sum_{[y]\in \u{p^{-1}(x)}} 
\frac{1}{|\Aut(\lambda)|} \,
\frac{1}{|\Aut(y)|}\]
which is equal to $|\Lambda|\,|p^{-1}(x)|$, as desired.
\end{proof}

\begin{proposition}
\label{scalarmult1}
Given a groupoid $\Lambda$ and a groupoid over $X$, 
say $v \maps \Psi \to X$, the groupoid
$\Lambda \times \Psi$ over $X$ satisfies
\[\utilde{\Lambda \times \Psi} = |\Lambda|\utilde{\Psi}.\]
\end{proposition}
\begin{proof} This follows from Lemma \ref{essinverseproduct}.
\end{proof}

\begin{proposition}
\label{scalarmult2}
Given a tame groupoid $\Lambda$ and a tame span
\[
\xymatrix{
& S\ar[dl] \ar[dr] & \\
Y & & X
}
\]
then $\Lambda \times S$ is tame and
\[\utilde{\Lambda \times S} = |\Lambda| \, \utilde{S}.\]
\end{proposition}
\begin{proof} This follows from Lemma \ref{essinverseproduct}.
\end{proof}

\subsection{Functoriality of Degroupoidification}

In this section we prove that our process of turning groupoids into
vector spaces and spans of groupoids into linear operators is indeed a
functor. We first show that the process preserves identities, then
show associativity of composition, from which many other things
follow, including the preservation of composition.  The lemmas in this
section add up to a proof of the following theorem:

\begin{thm}\label{functor}
Degroupoidification is a functor from the category of groupoids
and equivalence classes of tame spans to the category of real
vector spaces and linear operators.
\end{thm}

\begin{proof}
As mentioned above, the proof follows from Lemmas \ref{identities} and
\ref{composition}.
\end{proof}

\begin{lemma}\label{identities}
Degroupoidification preserves identities, i.e., given a groupoid $X$,
$\utilde{1_X} = 1_{\R^{\utilde{X}}}$, where $1_X$ is the identity span
from $X$ to $X$ and $1_{\R^{\utilde{X}}}$ is the identity operator on
$\R^{\utilde{X}}$.
\end{lemma}

\begin{proof}
This follows from the explicit formula given in Theorem \ref{matrix1}.
\end{proof}

We now want to prove the associativity of composition of tame spans.
Amongst the consequences of this proposition we can derive the
preservation of composition under degroupoidification.  Given a triple
of composable spans:
\[
\xymatrix{
 & T \ar[dl]_{q_T}\ar[dr]^{p_T} && S \ar[dl]_{q_S}\ar[dr]^{p_S} && R \ar[dl]_{q_R}\ar[dr]^{p_R} &\\
Z && Y && X && W
}
\]
we want to show that composing in the two possible orders---$T(SR)$
or $(TS)R$---will provide equivalent spans of groupoids.  In fact,
since groupoids, spans of groupoids, and isomorphism classes of maps
between spans of groupoids naturally form a bicategory, there exists a
natural isomorphism called the {\bf associator}.  This tells us that
the spans $T(SR)$ and $(TS)R$ are in fact equivalent.  But since we
have not constructed this bicategory, we will instead give an explicit
construction of the equivalence $T(SR) \stackrel{\sim}\rightarrow
(TS)R$.  

\begin{proposition}\label{associativity}
Given a composable triple of tame spans, the operation of composition
of tame spans by weak pullback is associative up to equivalence of
spans of groupoids.
\end{proposition}

\begin{proof}
We consider the above triple of spans in order to construct the
aforementioned equivalence.  The equivalence is simple to describe if
we first take a close look at the groupoids $T(SR)$ and $(TS)R$.  The
composite $T(SR)$ has objects $(t,(s,r,\alpha),\beta)$ such that $r\in
R$, $s\in S$, $t\in T$, $\alpha\maps q_R(r)\to p_S(s)$, and
$\beta\maps q_S(s)\to p_T(t)$, and morphisms $f \maps
(t,(s,r,\alpha),\beta) \to (t',(s',r',\alpha'),\beta')$, which consist
of a map $g \maps (s,r,\alpha) \to (s',r',\alpha')$ in $SR$ and a map
$h \maps t \to t'$ such that the following diagram commutes:
\[
\xymatrix{
q_S\pi_s((s,r,\alpha)) \ar[r]^---{\beta} \ar[d]_{q_S\pi_S(g)} & p_T(t) \ar[d]^{p_T(h)} \\
q_S\pi_s((s',r',\alpha')) \ar[r]_---{\beta'} & p_T(t')
}
\]
where $\pi_S$ maps the composite $SR$ to $S$.  Further, $g$ consists of a 
pair of maps $k\maps r \to r'$ and $j \maps s\to s'$ such that the following 
diagram commutes: 
\[
\xymatrix{
q_R(r) \ar[r]^{\alpha} \ar[d]_{q_S(k)} & p_S(s) \ar[d]^{p_S(j)} \\
q_R(r') \ar[r]_{\alpha'} & p_S(s')
}
\]

The groupoid $(TS)R$ has objects $((t,s,\alpha),r,\beta)$ such that
$r\in R$, $s\in S$, $t\in T$, $\alpha\maps q_S(s)\to p_T(t)$,
and $\beta\maps q_R(r)\to p_S(s)$, and morphisms $f \maps
((t,s,\alpha),r,\beta) \to ((t',s',\alpha'),r',\beta')$, which consist
of a map $g \maps (t,s,\alpha) \to (t',s',\alpha')$ in $TS$ and a map
$h \maps r \to r'$ such that the following diagram commutes:
\[
\xymatrix{
p_R(r) \ar[d]_{p_R(h)}\ar[r]^---{\beta} & p_S\pi_s((t,s,\alpha)) 
 \ar[d]^{p_S\pi_S(g)}   \\
p_R(r')  \ar[r]_---{\beta'} & p_S\pi_s((t',s',\alpha')) 
}
\]
Further, $g$ consists of a pair of maps $k\maps s \to s'$ and $j \maps
t\to t'$ such that the following diagram commutes:
\[
\xymatrix{
q_S(s) \ar[r]^{\alpha} \ar[d]_{q_S(k)} & p_T(t) \ar[d]^{p_T(j)} \\
q_S(s') \ar[r]_{\alpha'} & p_T(t')
}
\]

We can now write down a functor $F\maps T(SR) \to (TS)R$:
\[ (t,(s,r,\alpha),\beta) \mapsto ((t,s,\beta),r,\alpha) \]
Again, a morphism $f \maps (t,(s,r,\alpha),\beta) \to
(t',(s',r',\alpha'),\beta')$ consists of maps $k\maps r \to r'$, $j
\maps s\to s'$, and $h \maps t \to t'$.  We need to define $F(f) \maps
((t,s,\beta),r,\alpha) \to ((t',s',\beta'),r',\alpha')$.  The first
component $g' \maps (t,s,\beta) \to (t',s',\beta')$ consists of the
maps $j\maps s\to s'$ and $h\maps t \to t'$, and the following diagram
commutes:
\[
\xymatrix{
q_S(s) \ar[r]^{\beta}\ar[d]_{q_S(j)} & p_T(t) \ar[d]^{p_T(h)} \\
q_S(s') \ar[r]_{\beta'} & p_T(t')
}
\]
The other component map of $F(f)$ is $k \maps r \to r'$ and we see
that the following diagram also commutes:
\[
\xymatrix{
p_R(r) \ar[d]_{p_R(k)}\ar[rr]^---{\alpha} && p_S\pi_s((t,s,\beta))  \ar[d]^{p_S\pi_S(g')}   \\
p_R(r')  \ar[rr]_---{\alpha'} && p_S\pi_s((t',s',\beta')) 
}
\]
thus, defining a morphism in $(TS)R$.

We now just need to check that $F$ preserves identities and
composition and that it is indeed an isomorphism.  We will then have
shown that the apexes of the two spans are isomorphic.  First, given
an identity morphism $1\maps (t,(s,r,\alpha),\beta) \to
(t,(s,r,\alpha),\beta)$, then $F(1)$ is the identity morphism on
$((t,s,\beta),r,\alpha)$.  The components of the identity morphism are
the respective identity morphisms on the objects $r$,$s$, and $t$.  By
the construction of $F$, it is clear that $F(1)$ will then be an
identity morphism.

Given a pair of composable maps $f \maps (t,(s,r,\alpha),\beta) \to
(t',(s',r',\alpha'),\beta')$ and $f' \maps (t',(s',r',\alpha'),\beta')
\to (t'',(s'',r'',\alpha''),\beta'')$ in $T(SR)$, the composite is a
map $f'f$ with components $g'g \maps (s,r,\alpha) \to
(s'',r'',\alpha'')$ and $h'h\maps t \to t''$.  Further, $g'g$ has
component morphisms $k'k \maps r\to r''$ and $j'j \maps s\to s'$.  It
is then easy to check that under the image of $F$ this composition is
preserved.

The construction of the inverse of $F$ is implicit in the construction
of $F$, and it is easy to verify that each composite $FF^{-1}$ and
$F^{-1}F$ is an identity functor.  Further, the natural isomorphisms
required for an equivalence of spans can each be taken to be the
identity.
\end{proof}

It follows from the associativity of composition that degroupoidification
preserves composition:

\begin{lemma}
\label{composition}
Degroupoidification preserves composition.  That is, given a 
pair of composable tame spans:
\[
\xymatrix{
  & T\ar[dr]\ar[dl]  &   & S\ar[dr]\ar[dl] & \\
Z &                  & Y &                 & X
}
\]   
we have
\[ \utilde{T}\utilde{S} = \utilde{TS}. \]
\end{lemma}
 
\begin{proof}
Consider the composable pair of spans above along with a groupoid $\Psi$
over $X$:
\[
\xymatrix{
  & T\ar[dr]\ar[dl]  &   & S\ar[dr]\ar[dl] &  & \Psi\ar[dr]\ar[dl] &\\
Z &                  & Y &                 & X &     & 1
}
\] 
We can consider the groupoid over $X$ as a span by taking the right leg
to be the unique map to the terminal groupoid.  We can compose this triple
of spans  in two ways; either $T(S\Psi)$ or $(TS)\Psi$.  By the
Proposition \ref{associativity} stated above, these spans are equivalent.
By Theorem \ref{linops_equivspans1}, degroupoidification produces the same linear operators.
Thus, composition is preserved.  That is,
\[ \utilde{T}\utilde{S}\utilde{\Psi} = \utilde{TS}\utilde{\Psi}. \]
\end{proof}

\subsection{Inner Products and Adjoints}
\label{innerprodandadjoint}

Now we prove our results about the inner product of groupoids over
a fixed groupoid, and the adjoint of a span:

\begin{thm}\label{innerprod_theorem}
Given a groupoid $X$, there is a unique inner product $\ip{\cdot}{\cdot}$ 
on the vector space $L^2(X)$ such that
\[ \ip{\utilde{\Phi}}{\utilde{\Psi}} = |\ip{\Phi}{\Psi}| \]
whenever $\Phi$ and $\Psi$ are square-integrable groupoids over $X$.
With this inner product $L^2(X)$ is a real Hilbert space.
\end{thm}

\begin{proof}
Uniqueness of the inner product follows from the formula, since every vector 
in $L^2(X)$ is a finite-linear combination of vectors
$\utilde{\Psi}$ for square-integrable groupoids $\Psi$ over $X$.
To show the inner product exists, suppose that $\Psi_i, \Phi_i$ are
square-integrable groupoids over $X$ and $\alpha_i, \beta_i \in \R$
for $1 \le i \le n$.  Then we need to check that 
\[ \sum_{i}\alpha_i\utilde{\Psi}_i 
= \sum_{j}\beta_j \utilde{\Phi}_j = 0 \]
implies
\[ \sum_{i,j} \alpha_i\beta_j \, 
 |\langle \Psi_i, \Phi_j \rangle | = 0.\]
The proof here closely resembles the proof of existence in
Theorem \ref{PROCESS1}.  We leave to the reader the task of 
checking that $L^2(X)$ is complete in the norm corresponding 
to this inner product.
\end{proof}

\begin{proposition}\label{innerprod_adjoint}
Given a span
\[
\xymatrix{
& S \ar[dl]_{q} \ar[dr]^{p} & \\
Y & & X
}
\]
and a pair $v \maps \Psi \to X$, $w \maps \Phi \to Y$ of 
groupoids over $X$ and $Y$, respectively, there is an equivalence of groupoids
\[ \langle \Phi,S\Psi \rangle \simeq \langle S^\dagger\Phi,\Psi \rangle. \]
\end{proposition}
\begin{proof}
We can consider the groupoids over $X$ and $Y$ as spans with one leg
over the terminal groupoid $1$.  Then the result follows from the
equivalence given by associtativity in Lemma \ref{associativity} and Theorem
\ref{linops_equivspans1}.
Explicitly, $ \langle \Phi,S\Psi \rangle$ is the composite of spans $S\Psi$
and $\Phi$, while $\langle S^\dagger\Phi,\Psi \rangle$ is the composite
of spans $S^\dagger\Phi$ and $\Psi$.
\end{proof}

\begin{proposition}\label{adjoint_comp}
Given spans
\[
\xymatrix{
  & T \ar[dl]_{q_T} \ar[dr]^{p_T}  & & & S \ar[dl]_{q_S} \ar[dr]^{p_S} &\\
 Z & & Y & Y & & X 
}
\]
there is an equivalence of spans
\[(ST)^\dagger \simeq T^\dagger S^\dagger. \]
\end{proposition}

\begin{proof}
This is clear by the definition of composition.
\end{proof}

\begin{proposition}\label{adjoint_add}
Given spans
\[
\xymatrix{
& S \ar[dl]_{q_S} \ar[dr]^{p_S} & & & T \ar[dl]_{q_T} \ar[dr]^{p_T} & \\
Y & & X & Y & & X
}
\]
there is an equivalence of spans
\[(S+T)^\dagger \simeq S^\dagger + T^\dagger. \]
\end{proposition}

\begin{proof}
This is clear since the addition of spans is given by 
coproduct of groupoids.  This construction is symmetric
with respect to swapping the legs of the span.
\end{proof}

\begin{proposition}
\label{innerprodandadjointprops}
Given a groupoid $\Lambda$ and square-integrable
groupoids $\Phi$, $\Psi$, and $\Psi'$ over $X$, we have the
following equivalences of groupoids:
\begin{enumerate}
\item
\[\ip{\Phi}{\Psi} \simeq \ip{\Psi}{\Phi}.\] 
\item
\[\ip{\Phi}{\Psi + \Psi'} \simeq \ip{\Phi}{\Psi} + \ip{\Phi}{\Psi'}.\] 
\item
\[\ip{\Phi}{\Lambda \times \Psi} \simeq \Lambda \times \ip{\Phi}{\Psi}.\]
\end{enumerate}
\end{proposition}

\begin{proof}
Each part will follow easily from the definition of weak
pullback. First we label the maps for the groupoids over $X$ as
$v\maps\Phi\to X$, $w\maps\Psi\to X$, and $w' \maps\Psi' \to X$.

\begin{enumerate}
\item $\ip{\Phi}{\Psi} \simeq \ip{\Psi}{\Phi}.$\\ 
By definition of
weak pullback, an object of $\ip{\Phi}{\Psi}$ is a triple
$(a,b,\alpha)$ such that $a\in\Phi, b\in\Psi,$ and $\alpha\maps
v(a)\to w(b)$.  Similarly, an object of $\ip{\Psi}{\Phi}$ is a triple
$(b,a,\beta)$ such that $b\in\Psi, a\in\Phi,$ and $\beta\maps w(b) \to
v(a)$.  Since $\alpha$ is invertible, there is an evident equivalence
of groupoids.

\item $\ip{\Phi}{\Psi + \Psi'} \simeq \ip{\Phi}{\Psi} + \ip{\Phi}{\Psi'}.$\\
Recall that in the category of groupoids, the coproduct is just the
disjoint union over objects and morphisms. With this it is easy to
see that the definition of weak pullback will `split' over union. 

\item $\ip{\Phi}{\Lambda \times \Psi} \simeq \Lambda \times \ip{\Phi}{\Psi}.$\\
This follows from the associativity (up to isomorphism) of the 
cartesian product.
\end{enumerate}

\end{proof}

\subsubsection*{Acknowledgements}

We thank James Dolan, Todd Trimble, and the denizens of the
$n$-Category Caf\'e, especially Urs Schreiber and Bruce Bartlett, for
many helpful conversations.  AH was supported by the National Science
Foundation under Grant No.\ 0653646.  CW was supported by an FQXi
grant.

\appendix
\section{Review of Groupoids}\label{appendix}

\begin{definition}
A {\bf groupoid} is a category in which all morphisms are invertible.
\end{definition}

\begin{definition}
We denote the set of objects in a groupoid $X$ by $\Ob(X)$ and the
set of morphisms by $\Mor(X)$.
\end{definition}

\begin{definition}
A {\bf functor} $F \maps X \to Y$ between categories is a pair of
functions $F \maps \Ob(X) \to \Ob(Y)$ and $F \maps \Mor(X) \to
\Mor(Y)$ such that $F(1_x) = 1_{F(x)}$ for $x\in \Ob(X)$ and $F(gh) =
F(g)F(h)$ for $g,h\in \Mor(X)$.
\end{definition}

\begin{definition}
A {\bf natural transformation} $\alpha \maps F \To G$ between 
functors $F,G \maps X \to Y$ consists of a morphism $\alpha_x \maps F(x)
\to G(x)$ in $\Mor(Y)$ for each $x\in \Ob(X)$ such that for each
morphism $h\maps x \to x'$ in $\Mor(X)$ the following naturality
square commutes:
\[
\xymatrix{
F(x) \ar[r]^{\alpha_x} \ar[d]_{F(h)} & G(x) \ar[d]^{G(h)} \\
F(x') \ar[r]_{\alpha_{x'}} & G(x')
}
\]
\end{definition}

\begin{definition}
A {\bf natural isomorphism} is a natural transformation $\alpha \maps
F \To G$ between functors $F,G \maps X \to Y$ such that for each $x
\in X$, the morphism $\alpha_x$ is invertible.
\end{definition}
\noindent
Note that a natural transformation between functors between
{\it groupoids} is necessarily a natural isomorphism.  

In what follows, and throughout the paper, we write $x \in X$ as
shorthand for $x \in \Ob(X)$.  Also, several places throughout this
paper we have used the notation $\alpha\cdot F$ or $F\cdot\alpha$ to
denote operations combining a functor $F$ and a natural transformation
$\alpha$.  These operations are called `whiskering':

\begin{definition}
\label{whiskering}
Given groupoids $X$, $Y$ and $Z$, functors $F\maps X\to Y$, $G\maps
Y\to Z$ and $H\maps Y\to Z$, and a natural transformation $\alpha\maps
G \To H$, there is a natural transformation $\alpha\cdot F\maps
GF \To HF$ called the {\bf right whiskering} of $\alpha$ by
$F$.  This assigns to any object $x \in X$ the morphism
$\alpha_{F(x)}\maps G(F(x)) \to H(F(x))$ in $Z$, which we denote as
$(\alpha\cdot F)_{x}$.  Similarly, given a groupoid $W$ and a functor 
$J \maps Z \to W$, there is a natural transformation $J\cdot \alpha\maps
JG \To JH$ called the {\bf left whiskering} of $\alpha$ by $J$.  This 
assigns to any object $y \in Y$ the morphism $J(\alpha_y)\maps JG(y) 
\to JH(y)$ in $W$, which we denote as $(J\cdot\alpha)_{y}$.
\end{definition}

\begin{definition}
\label{equivalence_of_groupoids}
A functor $F \maps X \to Y$ between groupoids is called an {\bf
equivalence} if there exists a functor $G \maps Y \to X$, called the
{\bf weak inverse} of $F$, and natural isomorphisms $\eta \maps GF \To
1_X$ and $\rho \maps FG \To 1_Y$.  In this case we say $X$ and $Y$ are
{\bf equivalent}.
\end{definition}

\begin{definition}
A functor $F \maps X \to Y$ between groupoids is called {\bf faithful}
if for each pair of objects $x,y \in X$ the function $F \maps
\hom(x,y) \to \hom(F(x),F(y))$ is injective.
\end{definition}

\begin{definition}
A functor $F \maps X \to Y$ between groupoids is called {\bf full} if
for each pair of objects $x,y \in X$, the function $F \maps
\hom(x,y) \to \hom(F(x),F(y))$ is surjective.
\end{definition}

\begin{definition}
A functor $F \maps X \to Y$ between groupoids is called {\bf
essentially surjective} if for each object $y \in Y$, there
exists an object $x \in X$ and a morphism $f \maps F(x) \to y$ in
$Y$.
\end{definition}

\noindent
A functor has all three of the above properties if and only if the functor
is an equivalence.  It is often convenient to prove two groupoids are
equivalent by exhibiting a functor which is full, faithful and
essentially surjective.

\begin{definition}\label{MAP}
A {\bf map} from the span of groupoids
\[
\xymatrix{
& S\ar[dl]_{q} \ar[dr]^{p}   \\
Y & & X 
}
\]
to the span of groupoids
\[
\xymatrix{
& S'\ar[dl]_{q'} \ar[dr]^{p'}   \\
Y & & X 
}
\]
is a functor $F \maps S \to S'$ together with natural transformations
$\alpha \maps p \To p' F$, $\beta \maps q \To q' F$.
\end{definition}

\begin{definition}\label{EQUIVALENCE}
An {\bf equivalence} of spans of groupoids
\[
\xymatrix{
& S\ar[dl]_{q} \ar[dr]^{p} &  & & S'\ar[dl]_{q'} \ar[dr]^{p'} &\\
Y & & X & Y & & X
}
\]
is a map of spans $(F,\alpha,\beta)$ from $S$ to $S'$ together with a
map of spans $(G,\alpha',\beta')$ from $S'$ to $S$ and
natural isomorphisms $\gamma \maps GF \Rightarrow 1$ and $\gamma'
\maps FG \Rightarrow 1$ such that the following equations hold:
\[   \begin{array}{lll}
 1_p = (p \cdot \gamma)(\alpha'\cdot F) \alpha & &
 1_q = (q \cdot \gamma)(\beta'\cdot F) \beta \\
 1_{p'} = (p' \cdot \gamma')(\alpha\cdot G) \alpha' & &
 1_{q'} = (q' \cdot \gamma')(\beta\cdot G) \beta' .
\end{array}
\]
\end{definition}

\begin{lemma}\label{EQUIVGRPD}
Given equivalent groupoids $X$ and $Y$, $|X| = |Y|$.
\end{lemma}

\begin{proof}
From a functor $F \maps X \to Y$ between groupoids, we can
obtain a function $\u{F} \maps \u{X} \to 
\u{Y}$.  If $F$ is an equivalence, $\u{F}$
is a bijection.  Since these are the indexing sets for the sum 
in the definition of groupoid cardinality, we just need to check 
that for a pair of elements $[x] \in \u{X}$ and 
$[y] \in \u{Y}$ such that $\u{F}([x]) = [y]$, we
have $|\Aut(x)| = |\Aut(y)|$.  This follows from $F$ being full 
and faithful, and that the cardinality of automorphism
groups is an invariant of an isomorphism class of objects in a
groupoid.  Thus,
\[ |X| = 
\sum_{x \in \u{X}}\frac{1}{|\Aut(x)|} = 
\sum_{y \in \u{Y}}\frac{1}{|\Aut(y)|} = 
|Y|. \]
\end{proof}

\begin{lemma}\label{ESSENTIALPULLBACK}
Given a diagram of groupoids
\[ \def\objectstyle{\scriptstyle}
  \def\labelstyle{\scriptstyle}
   \xy
   (-20,10)*+{S}="1";
   (0,-10)*+{B}="2";
   (20,10)*+{T}="3";
        {\ar_{p} "1";"2"};
        {\ar^{q} "3";"2"};
        {\ar^{F} "1";"3"};
        {\ar@{=>}_<<{\scriptstyle \alpha} (-2,0); (2,4)};
\endxy
\]
where $F$ is an equivalence of groupoids, the
restriction of $F$ to the full inverse image $p^{-1}(b)$
\[ F|_{p^{-1}(b)} \maps p^{-1}(b) \to q^{-1}(b) \]
is an equivalence of groupoids, for any object $b\in B$.
\end{lemma}

\begin{proof}
It is sufficient to check that $F|_{p^{-1}(b)}$ is a full, faithful,
and essentially surjective functor from $p^{-1}(b)$ to $q^{-1}(b)$.
First we check that the image of $F|_{p^{-1}(b)}$ indeed lies in
$q^{-1}(b)$.  Given $b \in B$ and $x \in p^{-1}(b)$, there is
a morphism $\alpha_x \maps p(x) \to qF(x)$ in $B$.  Since $p(x) \in
[b]$, then $qF(x) \in [b]$.  It follows that $F(x) \in
q^{-1}(b)$.  Next we check that $F|_{p^{-1}(b)}$ is full and
faithful.  This follows from the fact that full inverse images are
full subgroupoids.  It is clear that a full and faithful functor
restricted to a full subgroupoid will again be full and faithful.  We
are left to check only that $F|_{p^{-1}(b)}$ is essentially
surjective.  Let $y \in q^{-1}(b)$.  Then, since $F$ is
essentially surjective, there exists $x \in S$ such that $F(x)
\in [y]$.  Since $qF(x) \in [b]$ and there is an isomorphism $\alpha_x
\maps p(x) \to qF(x)$, it follows that $x \in q^{-1}(b)$.  So
$F|_{p^{-1}(b)}$ is essentially surjective.  We have shown that
$F|_{p^{-1}(b)}$ is full, faithful, and essentially surjective, and,
thus, is an equivalence of groupoids.
\end{proof}

The data needed to construct a weak pullback of groupoids is a
`cospan':

\begin{definition}
Given groupoids $X$ and $Y$, a {\bf cospan} from $X$ to $Y$ is a diagram
\[\xymatrix{
Y\ar[dr]_g &  & X\ar[dl]^f\\
  & Z &  \\
}\]
where $Z$ is groupoid and $f\maps X\to Z$ and $g \maps Y\to Z$ are functors.
\end{definition}

\noindent
We next prove a lemma stating that the weak pullbacks of equivalent
cospans are equivalent.  Weak pullbacks, also called {\it iso-comma
objects}, are part of a much larger family of limits called {\it
flexible limits}.  To read more about flexible limits, see the work of
Street \cite{Street:1980} and Bird \cite{Bird:1989}.  A vastly more
general theorem than the one we intend to prove holds in this class of
limits.  Namely: for any pair of parallel functors $F,G$ from an
indexing category to $\Cat$ with a pseudonatural equivalence
$\eta\maps F \To G$, the pseudo-limits of $F$ and $G$ are equivalent.
But to make the paper self-contained, we strip this theorem down and
give a hands-on proof of the case we need.

To show that equivalent cospans of groupoids have equivalent weak
pullbacks, we need to say what it means for a pair of cospans to be
equivalent.  As stated above, this means that they are given by a pair
of parallel functors $F,G$ from the category consisting of a
three-element set of objects $\lbrace 1,2,3 \rbrace$ and two morphisms
$a\maps 1 \to 3$ and $b \maps 2 \to 3$.  Further there is a
pseudonatural equivalence $\eta \maps F \to G$.  In simpler terms,
this means that we have equivalences $\eta_i \maps F(i) \to G(i)$ for 
$i = 1,2,3$, and squares commuting up to natural isomorphism:
\[
 \def\objectstyle{\scriptstyle}
  \def\labelstyle{\scriptstyle}
   \xy
   (-30,10)*+{F(1)}="1";
   (-30,-10)*+{F(3)}="2";
   (-10,10)*+{G(1)}="3";
   (-10,-10)*+{G(3)}="4";
   (10,10)*+{F(1)}="5";
   (10,-10)*+{F(3)}="6";
   (30,10)*+{G(1)}="7";
   (30,-10)*+{G(3)}="8";   
        {\ar_{\eta_1} "1";"2"};
        {\ar^{F(a)} "1";"3"};
        {\ar^{\eta_3} "3";"4"};
        {\ar_{G(a)} "2";"4"};
        {\ar_{\eta_2} "5";"6"};
        {\ar^{F(b)} "5";"7"};
        {\ar^{\eta_3} "7";"8"};
        {\ar_{G(b)} "6";"8"};
        {\ar@{=>}^{v} (-23,-3)*{}; (-17,3)*{}}; 
        {\ar@{=>}^{w} (17,-3)*{}; (23,3)*{}}; 
\endxy
\]

For ease of notation we will consider the equivalent cospans:
\[
\xymatrix{
Y\ar[dr]_{g} && X\ar[dl]^{f} & \hat{Y}\ar[dr]_{\hat{g}} && \hat{X}\ar[dl]^{\hat{f}} \\
&Z& & &\hat{Z}&
}
\]
with equivalences $\hat{x}\maps X \to \hat{X}$, $\hat{y}\maps Y \to
\hat{Y}$, and $\hat{z}\maps Z \to \hat{Z}$ and natural isomorphisms
$v\maps \hat{z}f \Rightarrow \hat{f}\hat{x}$ and $w \maps \hat{z}g
\Rightarrow \hat{g}\hat{y}$.

\begin{lemma}\label{skeletal}
Given equivalent cospans of groupoids as described above, the weak
pullback of the cospan
\[
\xymatrix{
Y\ar[dr]_{g} & & X\ar[dl]^{f} \\
& Z & 
}
\]
is equivalent to the weak pullback of the cospan
\[
\xymatrix{
\hat{Y}\ar[dr]_{\hat{g}} & & \hat{X}\ar[dl]^{\hat{f}} \\
& \hat{Z} & 
}
\]
\end{lemma}

\begin{proof}
We construct a functor $F$ between the weak
pullbacks $XY$ and $\hat{X}\hat{Y}$ and show that this functor is an
equivalence of groupoids, i.e., that it is full, faithful and
essentially surjective.  We recall that an object in the weak pullback
$XY$ is a triple $(r,s,\alpha)$ with $r\in X$, $s\in Y$ and
$\alpha \maps f(r) \to g(s)$.  A morphism in
$\rho \maps (r,s,\alpha) \to (r',s',\alpha')$ in $XY$ is given by a
pair of morphisms $j \maps r\to r'$ in $X$ and $k\maps s\to s'$ in $Y$
such that $g(k)\alpha = \alpha' f(j)$.  We define
\[F \maps XY \to \hat{X}\hat{Y}\]
on objects by
\[(r,s,\alpha) \mapsto (\hat{x}(r),\hat{y}(s),w_s^{-1}\hat{z}(\alpha)v_r)\]
and on a morphism $\rho$ by sending $j$ to $\hat{x}(j)$ and $k$ to
$\hat{y}(k)$.  To check that this functor is well-defined we consider
the following diagram:
\[
\xymatrix{
\hat{f}\hat{x}(r)\ar[r]^{v_r}\ar[d]_{\hat{f}\hat{x}(j)} & \hat{z}f(r)\ar[r]^{\hat{z}(\alpha)}\ar[d]_{\hat{z}f(j)} & \hat{z}g(s)\ar[r]^{w_s^{-1}}\ar[d]^{\hat{z}g(k)} & \hat{g}\hat{y}(s)\ar[d]^{\hat{g}\hat{y}(k)}\\
\hat{f}\hat{x}(r')\ar[r]_{v_{r'}} & \hat{z}f(r')\ar[r]_{\hat{z}(\alpha')} & \hat{z}g(s')\ar[r]_{w_{s'}^{-1}} & \hat{g}\hat{y}(s')
}
\]
The inner square commutes by the assumption that $\rho$ is a morphism
in $XY$.  The outer squares commute by the naturality of $v$ and $w$.
Showing that $F$ respects identities and composition is
straightforward.

We first check that $F$ is faithful.  Let $\rho,\sigma \maps
(r,s,\alpha) \to (r',s',\alpha')$ be morphisms in $XY$ such that
$F(\rho) = F(\sigma)$.  Assume $\rho$ consists of
morphisms $j\maps r \to r'$, $k\maps s \to s'$ and $\sigma$ consists
of morphisms $l\maps r \to r'$ and $m \maps s \to s'$.  It follows
that $\hat{x}(j) = \hat{x}(l)$ and $\hat{y}(k) = \hat{y}(m)$.  Since
$\hat{x}$ and $\hat{y}$ are faithful we have that $j=l$ and $k=m$.
Thus, we have shown that $\rho = \sigma$ and $F$ is
faithful.

To show that $F$ is full, we assume $(r,s,\alpha)$ and
$(r',s',\alpha')$ are objects in $XY$ and $\rho \maps
(\hat{x}(r),\hat{y}(s),\hat{z}(\alpha)) \to
(\hat{x}(r'),\hat{y}(s'),\hat{z}(\alpha'))$ is a morphism in
$\hat{X}\hat{Y}$ consisting of morphisms $j \maps \hat{x}(r) \to
\hat{x}(r')$ and $k \maps \hat{y}(s) \to \hat{y}(s')$.  Since
$\hat{x}$ and $\hat{y}$ are full, there exist morphisms $\tilde{j}
\maps r \to r'$ and $\tilde{k}\maps s \to s'$ such that
$\hat{x}(\tilde{j}) = j$ and $\hat{y}(\tilde{k}) = k$.  We consider
the following diagram:
\[
\xymatrix{
\hat{z}(f(r))\ar[r]^{v_r^{-1}} \ar[d]_{\hat{z}(f(\tilde{j}))} & \hat{f}\hat{x}(r)\ar[r]^{\hat{z}(\alpha)}\ar[d]_{\hat{f}\hat{x}(\tilde{j})} & \hat{g}\hat{y}(s)\ar[r]^{w_s}\ar[d]^{\hat{g}\hat{y}(\tilde{k})} & \hat{z}(g(s)) \ar[d]^{\hat{z}(g(\tilde{k}))} \\
\hat{z}(f(r'))\ar[r]_{v_{r'}^{-1}} & \hat{f}\hat{x}(r') \ar[r]_{\hat{z}(\alpha')} & \hat{g}\hat{y}(s')\ar[r]_{w_s} & \hat{z}(g(s'))
}
\]
The center square commutes by the assumption that $\rho$ is a morphism
in $\hat{X}\hat{Y}$, and the outer squares commute by naturality of
$v$ and $w$.  Since $\hat{z}$ is full, there exists morphisms
$\bar{\alpha}\maps f(r) \to g(s)$ and $\bar{\alpha'} \maps f(r') \to
g(s')$ such that $\hat{z}(\bar{\alpha}) = w_s\hat{z}(\alpha)v_r^{-1}$
and $\hat{z}(\bar{\alpha'}) = w_{s'}\hat{z}(\alpha')v_{r'}^{-1}$.  Now
since $\hat{z}$ is faithful, we have that
\[
\xymatrix{
f(r)\ar[r]^{\bar{\alpha}} \ar[d]_{f(\tilde{j})} & g(s) \ar[d]^{g(\tilde{k})} \\
f(r') \ar[r]_{\bar{\alpha'}} & g(s')
}
\]
commutes.  Hence, $F$ is full.

To show $F$ is essentially surjective we let $(r,s,\alpha)$
be an object in $\hat{X}\hat{Y}$.  Since $\hat{x}$ and $\hat{y}$ are
essentially surjective, there exist $\tilde{r} \in X$ and
$\tilde{s}\in Y$ with isomorphisms $\beta \maps
\hat{x}(\tilde{r}) \to r$ and $\gamma \maps \hat{y}(\tilde{s}) \to s$.
We thus have the isomorphism:
\[\hat{z}(f(\tilde{r}))\stackrel{v_{\tilde{r}^{-1}}}\longrightarrow \hat{f}(\hat{x}(\tilde{r}))\stackrel{\hat{f}(\beta)}\longrightarrow \hat{f}(r) \stackrel{\alpha}\longrightarrow \hat{g}(s) \stackrel{\hat{g}(\gamma^{-1})}\longrightarrow \hat{g}(\hat{y}(\tilde{s}))\stackrel{w_{\tilde{s}}}
\longrightarrow \hat{z}(g(\tilde{s}))\]
Since $\hat{z}$ is full, there exists an isomorphism $\mu \maps
f(\tilde{r}) \to g(\tilde{s})$ such that $\hat{z}(\mu) =
w_s\hat{g}(\gamma^{-1})\alpha\hat{f}(\beta)v_r^{-1}$.  We have
constructed an object $(\tilde{r},\tilde{s},\mu)$ in $XY$ and we need
to find an isomorphism from $F((\tilde{r},\tilde{s},\mu) =
(\hat{x}(\tilde{r}), \hat{y}(\tilde{s}),w_s^{-1}\hat{z}(\mu)v_r)$ to
$(r,s,\alpha)$.  This morphism consists of $\beta \maps
\hat{x}(\tilde{r}) \to r$ and $\gamma \maps \hat{y}(\tilde{s}) \to s$.
That this is an isomorphism follows from $\beta,\gamma$ being
isomorphisms and the following calculation:
\begin{eqnarray*}
\hat{g}(\gamma)w_s^{-1}\hat{z}(\mu)v_r &=& \hat{g}(\gamma)w_{\tilde{s}}^{-1}w_{\tilde{s}}\hat{g}(\gamma^{-1})\alpha\hat{f}(\beta)v_{\tilde{r}}^{-1}v_{\tilde{r}}\\
&=& \alpha\hat{f}(\beta)
\end{eqnarray*}
We have now shown that $F$ is essentially surjective, and
thus an equivalence of groupoids.
\end{proof}


\begin{thebibliography}{10}

\bibitem{AM:2008} M.\ Aguiar and S.\ Mahajan, Monoidal functors,
species and Hopf algebras, available at
\href{http://www.math.tamu.edu/~maguiar/a.pdf}
{$\langle$http:/!/www.math.tamu.edu/$\sim$maguiar/a.pdf$\rangle$}.

\bibitem{Baez}
J.\ Baez, Groupoidification.  Available at\hfill\break
\href{http://math.ucr.edu/home/baez/groupoidification/}
{$\langle$http:/\!/math.ucr.edu/home/baez/groupoidification/$\rangle$}.

\bibitem{BaezDolan:1998}
J.\ Dolan and J.\ Baez, Categorification, in
{\sl Higher Category Theory,} eds.\ E.\ Getzler and M.\ Kapranov,
Contemp.\ Math.\ 230, American Mathematical Society, Providence, Rhode
Island, 1998, pp.\ 1--36.  Also available as
\href{http://arxiv.org/abs/math/9802029}
{arXiv:math/9802029}.

\bibitem{BaezDolan:2001}
J.\ Baez and J.\ Dolan, From finite sets to Feynman diagrams, in {\em
Mathematics Unlimited---2001 and Beyond}, eds.\ B.\ Engquist
and W.\ Schmid, Springer, Berlin, 2001, pp.\ 29--50.  
Also available as \href{http://arxiv.org/abs/math/0004133}
{arXiv:math/0004133}.

\bibitem{Behrends:1993}
K.\ Behrend, The Lefschetz trace formula for algebraic stacks, {\em
Invent.\ Math.\ }{\bf 112} (1993), 127--149.

\bibitem{BLL}
F.\ Bergeron, G.\ Labelle, and P.\ Leroux, {\sl Combinatorial Species
and Tree-Like Structures}, Cambridge U.\ Press, Cambridge, 1998.

\bibitem{BFK} 
J.\ Bernstein, I.\ Frenkel and M.\ Khovanov, A categorification
of the Temperley--Lieb algebra and Schur quotients of $U(\sl_2)$
via projective and Zuckerman functors, \textrm{Selecta Math.\ (N.\ S.)}
{\bf 5} (1999), 199--241.
Also available as \href{http://arxiv.org/abs/math/0002087}
{arXiv:math/0002087}.

\bibitem{Bird:1989}
G.\ Bird, G.\ Kelly, A.\ Power, and R.\ Street, Flexible limits for
$2$-categories, {\sl J. Pure Appl. Algebra }{\bf 61} (1989),
1--27.

\bibitem{Borel} A.\ Borel, {\sl Linear Algebraic Groups}, Springer, 
Berlin, 1991.

\bibitem{Brown} K.\ Brown, {\sl Buildings}, Springer, Berlin,
1989.

\bibitem{CF}
L.\ Crane and I.\ Frenkel, Four-dimensional topological
quantum field theory, Hopf categories, and the canonical bases,
\textsl{Jour.\ Math.\ Phys.} \textbf{35} (1994), 5136--5154.
Also available as \href{http://arxiv.org/abs/hep-th/9405183}
{arXiv:hep-th/9405183}

\bibitem{ElKh}
B.\ Elias and M.\ Khovanov, Diagrammatics for Soergel categories.
Available as \href{http://arxiv.org/abs/0902.4700}{arXiv:0902.4700}.

\bibitem{FioreLeinster:2005}
M.\ Fiore and T.\ Leinster, Objects of categories as complex numbers,
{\sl Adv.\ Math.\ }{\bf 190} (2005), 264--277.  Also available as
\href{http://arxiv.org/abs/math/0212377}{arXiv:math/0212377}.

\bibitem{Gabriel}
P.\ Gabriel, Unzerlebare Darstellungen I, \textit{Manuscr.\ Math.\ }
{\bf 6} (1972), 71--103.

\bibitem{Garrett} P.\ Garrett, {\sl Buildings and Classical Groups},
CRC Press, New York, 1997.   Preliminary version available at \hfill \break
\href{http://www.math.umn.edu/~garrett/m/buildings/}
{$\langle$http:/\!/www.math.umn.edu/$\sim$garrett/m/buildings/$\rangle$}.

\bibitem{GJ}
P.\ G.\ Goerss and J.\ F.\ Jardine, Simplicial Homotopy Theory,
Birk\"auser, Basel, 1999.

\bibitem{GM:2002} M.\ Guta and H.\ Maassen, Symmetric Hilbert spaces
arising from species of structures, {\sl Math.\ Zeit.\ }{\bf 239}
(2002), 477--513.  Also available as
\href{http://arxiv.org/abs/math-ph/0007005}
{arXiv:math-ph/0007005}.

\bibitem{Hoffnung} A.\ E.\ Hoffnung, The Hecke bicategory.  
Available as \href{http://arxiv.org/abs/1007.1931}{arXiv:1007.1931}.

\bibitem{Hubery} A.\ W.\ Hubery, Ringel--Hall algebras, lecture
notes available at
\href{http://www.maths.leeds.ac.uk/~ahubery/RHAlgs.html}
{$\langle$http:/\!/www.maths.leeds.ac.uk/~ahubery/RHAlgs.html$\rangle$}.

\bibitem{Humphreys} J.\ Humphreys, {\sl Reflection Groups and
Coxeter Groups}, Cambridge U.\ Press, Cambridge, 1992.

\bibitem{Jones} V.\ Jones, Hecke algebra representations of braid
groups and link polynomials, {\sl Ann.\ Math.\ }{\bf 126} (1987), 335--388.

\bibitem{Joyal} A.\ Joyal, Une th\'eorie combinatoire des s\'eries
formelles, {\sl Adv.\ Math.\ }{\bf 42} (1981), 1--82.

\bibitem{Joyal2} A.\ Joyal, Foncteurs analytiques et esp\`eces des
structures, in Springer Lecture Notes in Mathematics {\bf 1234},
Springer, Berlin, 1986, pp.\ 126--159.

\bibitem{KaLu} D.\ Kazhdan and G.\ Lusztig, Representations of
Coxeter groups and Hecke algebras, \emph{Invent. Math.} {\bf 53} 
(1979), 165--184.

\bibitem{Kh}
M.\ Khovanov, A categorification of the Jones polynomial, 
\textsl{Duke Math.\ J.\ }{\bf 101} (2000), 359--426.  
Also available as \href{http://arxiv.org/abs/math/9908171}
{arXiv:math/9908171}.

\bibitem{Kh2}
M.\ Khovanov, Categorifications from planar diagrammatics.
Available as 
\href{http://arxiv.org/abs/1008.5084}{arXiv:1008.5084}.

\bibitem{KL} 
M.\ Khovanov and A.\ D.\ Lauda, A diagrammatic approach to
categorification of quantum groups I, {\em Represent.\ Theory} {\bf
13} (2009), 309--347.  Also available as
\href{http://arxiv.org/abs/0803.4121}{arXiv:0803.4121}.

A diagrammatic approach to categorification of quantum groups II.
Also available as
\href{http://arxiv.org/abs/0804.2080}{arXiv:0804.2080}.

A diagrammatic approach to categorification of quantum groups III.
Also available as
\href{http://arxiv.org/abs/0807.3250}{arXiv:0807.3250}.

\bibitem{Kim:1995}
M.\ Kim, A Lefschetz trace formula for equivariant cohomology, {\em
Ann.\ Sci.\ \'Ecole Norm.\ Sup.\ (4) }{\bf 28} (1995), no. 6, 669--688.
Also available at
\href{http://www.numdam.org/numdam-bin/fitem?id=ASENS_1995_4_28_6_669_0}
{$\langle$http:/\!/www.numdam.org$\rangle$}.

\bibitem{Leinster:2008}
T.\ Leinster, The Euler characteristic of a category, {\em
Doc. Math.\ }{\bf 13} (2008), 21--49.  Also available as
\href{http://arxiv.org/abs/math/0601458}
{arXiv:math/0610260}.

\bibitem{Mac Lane} S.\ Mac Lane, {\sl Categories for the
Working Mathematician}, Springer, Berlin, 1998.

\bibitem{Mathas}
A.\ Mathas, {\sl Iwahori-Hecke Algebras and Schur Algebras of the Symmetric 
Ggroup}, American Mathematical Society, Providence, Rhode Island, 1999.

\bibitem{McCrudden}
Paddy McCrudden, Balanced coalgebroids,
\textsl{Th.\ Appl.\ Cat.\ }{\bf 7} (2000), 71--147. Available at
\hfill \break
\href{http://www.tac.mta.ca/tac/volumes/7/n6/7-06abs.html}
{$\langle$http:/\!/www.tac.mta.ca/tac/volumes/7/n6/7-06abs.html$\rangle$}.

\bibitem{Morton:2006}
J.\ Morton, Categorified algebra and quantum mechanics, {\em Theory
Appl. Categ.}{\bf 16}, (2006), 785-854.  Also available as
\href{http://arxiv.org/abs/math/0601458}
{arXiv:math/0601458}.

\bibitem{Ringel} C.\ Ringel, Hall algebras and quantum groups,
\textsl{Invent.\ Math.\ } {\bf 101} (1990), 583--591.

\bibitem{Ringel2} C.\ Ringel, Hall algebras revisited,
\textsl{Israel Math.\ Conf.\ Proc.\ }{\bf 7} (1993), 171--176.
Also available at \hfill \break
\href{http://www.mathematik.uni-bielefeld.de/~ringel/opus/hall-rev.pdf}
{$\langle$http:/\!/www.mathematik.uni-bielefeld.de/$\sim$ringel/opus/hall-rev.pdf$\rangle$}.

\bibitem{Rou} R.\ Rouquier, Categorification of the braid groups.  
Available as \href{http://arxiv.org/abs/math/0409593}{arXiv:math/0409593}.

\bibitem{Rou2} R.\ Rouquier, 2-Kac--Moody algebras.
Available as \href{http://arxiv.org/abs/0812.5023}{arXiv:0812.5023}.

\bibitem{Stroppel} C.\ Stroppel, Categorification of the Temperley--Lieb
category, tangles and cobordisms via projective functors, 
\textsl{Duke Math.\ J.\ }{\bf 126} (2005), 547--596. 

\bibitem{Urs}
U.\ Schreiber, Span trace, $n$Lab article available at \hfill \break
\href{http://ncatlab.org/nlab/show/span+trace}
{$\langle$http:/\!/ncatlab.org/nlab/show/span+trace$\rangle$}.

\bibitem{Street:1980}
R. Street, Fibrations in bicategories, {\em Cahiers Top.\
G\'eom. Diff.\ }{\bf 21} (1980), 111--160.

\bibitem{Schiffman}
O.\ Schiffman, Lectures on Hall algebras, available as
\href{http://arxiv.org/abs/math/0611617}{arXiv:math/0611617}.

\bibitem{Soergel} W.\ Soergel, The combinatorics of Harish-Chandra
bimodules, \emph{J.\ Reine Angew.\ Math.\ }{\bf 429}, (1992) 49--74.

\bibitem{Szczesny} M.\ Szczesny, Representations of quivers over
$\F_1$.  Available as \href{http://arxiv.org/abs/1006.0912}{arXiv:1006.0912}.

\bibitem{WW} B.\ Webster and G.\ Williamson, A geometric model 
for Hochschild homology of Soergel bimodules, \emph{Geom.
Top.\ } {\bf 12} (2008), 1243--1263.  Also available as 
\href{http://arxiv.org/abs/0707.2003}{arXiv:0707.2003}.

\bibitem{Weinstein:2008}
A.\ Weinstein, The volume of a differentiable stack,
to appear in the proceedings of Poisson 2008 (Lausanne, July 2008).
Available as \href{http://arxiv.org/abs/0809.2130}
{arXiv:0809.2130}.

\bibitem{Wilf} H.\ Wilf, {\sl Generatingfunctionology}, Academic
Press, Boston, 1994.  Also available at
\href{http://www.cis.upenn.edu/~wilf/}
{$\langle$http:/\!/www.cis.upenn.edu/$\sim$wilf/$\rangle$}.

\end{thebibliography}
\end{document}